\documentclass[11pt]{amsart}
\usepackage{palatino}
\usepackage{graphicx}
\usepackage{geometry}
\usepackage{hyperref}
\usepackage{soul,color}
\usepackage{enumerate}
\usepackage{pstricks}
\usepackage{paralist, amsthm,amsfonts,amsmath,amssymb,graphics}
\newcommand{\footnotetextplain}[1]{\begingroup\def\@thefnmark{}%
  \long\def\@makefntext##1{\parindent 0pt\noindent ##1}\@footnotetext{#1}
  \endgroup}

\makeatletter 

\newcommand{\TeoremaAmbFinalMarcat}[1]{%
  \expandafter\gdef\csname end#1\endcsname{\@endtheorem}}
  
               {\hfill\rule{2.5mm}{2.5mm} \vspace{\parskip} } 
\newtheorem{theorem}{Theorem}[section]
\newtheorem{proposition}[theorem]{Proposition}
\newtheorem{corollary}[theorem]{Corollary}
\newtheorem{lemma}[theorem]{Lemma}

\theoremstyle{definition}
\newtheorem{definition}[theorem]{Definition} \TeoremaAmbFinalMarcat{defi}
 \TeoremaAmbFinalMarcat{ex}
\newtheorem{remark}[theorem]{Remark} \TeoremaAmbFinalMarcat{rem}
 \TeoremaAmbFinalMarcat{conv}

\newenvironment{proclama}[1]{
                \par\vspace{\topsep}\noindent{\bf #1}
                \begin{em}}
                {\end{em}\par\vspace{\topsep}}

\def\@enum@{\list{\csname label\@enumctr\endcsname}%
           {\usecounter{\@enumctr}\def\makelabel##1{\hss\llap{##1}}
           \itemsep=2pt\parsep=0pt\topsep=3pt plus 1pt minus 1 pt}}

\newenvironment{alphlist}{\enumerate[(a)]}{\endenumerate}

\newenvironment{numlist}{\enumerate[(1)]}{\endenumerate}

\newcommand{\xnorm}[1]{ \Vert #1 \Vert }

\newcommand{\zz}[1]{\mathbb #1}

\newcommand{\gen}{\mathcal{G}}
\newcommand{\ca}{g}
\newcommand{\geod}{\mathbb{G}}

\newcommand{\cov}{\mathrm{cov}}
\newcommand{\var}{\mathrm{Var}}

\title{Self-intersections in combinatorial topology: statistical structure}
\author{Moira Chas}
\address{ Stony Brook University\\
Department of Mathematics\\ Stony Brook, NY, 11794.}
\author{Steven P. Lalley} \address{University of Chicago\\ Department
of Statistics \\ 5734
University Avenue \\
Chicago IL 60637.}
\email{moira@math.sunysb.edu, lalley@galton.uchicago.edu}

\date{\today}
\subjclass{Primary 57M05, secondary 53C22, 37D40}
\keywords{closed curves, surfaces, self-intersection, central limit theorem,
Markov chain, U-statistic}
\thanks{Supported by NSF grant DMS  - 0805755}
\makeindex

\begin{document}

\begin{abstract}
Oriented closed curves on an orientable surface with boundary are
described up to continuous deformation by reduced cyclic words in the
generators of the fundamental group and their inverses. By
self-intersection number one means the minimum number of transversal
self-intersection points of representatives of the class. We prove
that if a class is chosen at random from among all classes of $m$
letters, then for large $m$ the distribution of the self-intersection
number approaches the Gaussian distribution.

\end{abstract}

\maketitle
\newpage
\tableofcontents

\section{Introduction}\label{sec:introduction} 


Oriented closed curves in a surface with boundary are, up to
continuous deformation, described by reduced cyclic words in a set of
free generators of the fundamental group and their inverses. (Recall
that such words represent the conjugacy classes of the fundamental
group.) Given a reduced cyclic word $\alpha$, define the
\emph{self-intersection number} $N (\alpha)$ to be the minimum number
of transversal double points among all closed curves represented by
$\alpha$. (See Figure~\ref{example}.)
\begin{figure}[http]
\begin{pspicture}(12,3)
  \psset{xunit=0.4,yunit=0.40}\psdot(4,2.5)\psdot(8,2.5)
\psccurve[showpoints=false,linecolor=gray,linewidth=1.pt] (0,2.5)
(4,4) (5,2.5) (4,1.5) (2,2.5) (4,3.5) (6,2.5) (8,1.5) (11,2.5) (8,3.5)
(7,2.5) (8,.5) (10,2.5)  (8,4) (4,0)

\psset{origin={18.0,0.0}}
  \psdot(4,2.5)\psdot(8,2.5)
\psccurve[showpoints=false,linecolor=gray,linewidth=1.pt] (0,2.5)
(4,4) (5,2.5) (4,1.5) (2,2.5) (4,3.5) (6,2.5) (8,1.5) (9,2.5) (8,3.5)
(7,2.5) (8,.5) (10,2.5)  (8,4) (4,0)
 \end{pspicture}
\caption{ Two representatives of $aa\bar{b}\bar{b}$ in the doubly punctured plane. The second curve
has fewest self-intersections in its free homotopy class.}\label{example}
\end{figure}
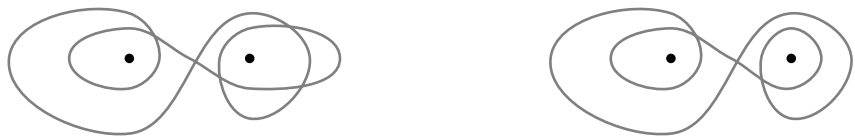
 Fix a positive integer $n$ and consider how the self-intersection
number $N (\alpha)$ varies over the population $\mathcal{F}_{n}$ of
all reduced cyclic words of length $n$. The value of $N (\alpha)$ can
be as small as $0$, but no larger than $O(n^{2})$. See
\cite{chas-phillips:2010}, \cite{chas-phillips:preprint} for precise
results concerning the maximum of $N (\alpha)$ for $\alpha \in
\mathcal{F}_{n}$, and \cite{mirzakhani} for
sharp  results on the related problem of determining the growth of the
number of non self-intersecting closed geodesics up to a given length relative to
a hyperbolic metric.

For small values of $n$, using algorithms in \cite{cohen-lustig}, or  \cite{chas:2004} it is computationally feasible to compute the
self-intersection counts $N (\alpha)$ for all words $\alpha \in
\mathcal{F}_{n}$. Such computations show that, even for relatively
small $n$, the distribution of $N(\alpha)$ over $\mathcal{F}_{n}$ is
very nearly Gaussian. (See Figure~\ref{histogram}.)
\begin{figure}[htp]\centering
\includegraphics[width=10cm,angle=-90]{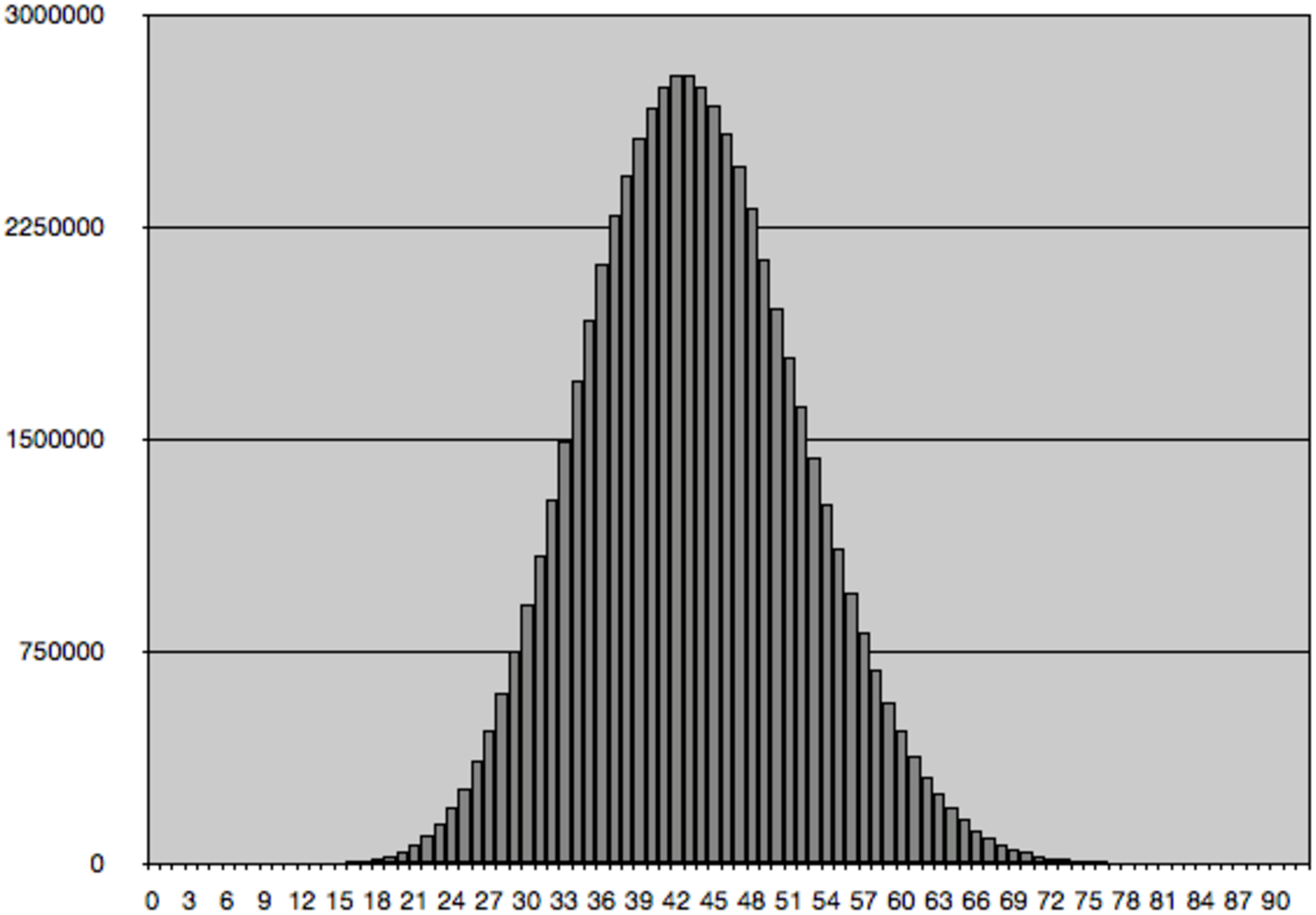} 
\caption{A histogram showing the distribution of self-intersection
numbers over  all reduced cyclic words of length 19 in the doubly punctured plane. The horizontal
coordinate shows the  self-intersection count $k$; the vertical
coordinate shows the  number of cyclic reduced words for which the
self-intersection number is  $k$. }\label{histogram}    
\end{figure}
The purpose of this paper is to prove that as $n \rightarrow \infty$
the distribution of $N (\alpha)$ over the population
$\mathcal{F}_{n}$, suitably scaled, does indeed approach a Gaussian
distribution:

\begin{proclama}{Main Theorem.} \label{theorem:main}
Let $\Sigma$ be an orientable, compact
surface with boundary and negative Euler characteristic $\chi$, and
set 
\begin{equation}\label{eq:constants}
\kappa =\kappa_{\Sigma }=\frac{\chi}{3(2\chi-1)} \mbox{ and }
\sigma^2=\sigma^{2}_{\Sigma}=
=\frac{2 \chi  (2 \chi ^2-2 \chi +1)}{45(2 \chi-1 )^2 (\chi -1)}
\end{equation}
Then for any $a<b$ the proportion of words $\alpha \in
\mathcal{F}_{n}$ such that 
\[
	a<\frac{N (\alpha)-\kappa n^{2}}{n^{3/2}}<b
\] 
converges, as $n \rightarrow \infty$, to
\[
	\frac{1}{\sqrt{2\pi}\sigma } \int_{a}^{b} \exp
	\{-x^{2}/2\sigma^{2} \} \,dx.
\]
\end{proclama}

Observe that the limiting variance $\sigma^{2}$ is \emph{positive} {
if the Euler characteristic is negative}. Consequently, the theorem
implies that (i) for most words $\alpha \in \mathcal{F}_{n}$ the
self-intersection number $N (\alpha)$ is to first order
well-approximated by $\kappa n^{2}$; and (ii) typical variations of $N
(\alpha)$ from this first-order approximation (``fluctuations'') are
of size $n^{3/2}$.

It is relatively easy to understand (if not to prove) why the number
of self-intersections of typical elements of $\mathcal{F}_{n}$ should
grow like $n^{2}$. Following is a short heuristic argument: Consider
the lift of a closed curve with minimal self-intersection number in
its class to the universal cover of the surface $\Sigma$. This lift
will cross $n$ images of the fundamental polygon, where $n$ is the
corresponding word length, and these crossings can be used to
partition the curve into $n$ nonoverlapping segments in such a way
that each segment makes one crossing of an image of the fundamental
polygon. The self-intersection count for the curve is then the number
of \emph{pairs} of these segments whose images in the fundamental
polygon cross. It is reasonable to guess that for typical classes
$\alpha \in \mathcal{F}_{n}$ (at least when $n$ is large) these
segments look like a \emph{random} sample from the set of all such
segments, and so the law of large numbers then implies that the number
of self-intersections should grow like $n^{2}\kappa '/2$ where $\kappa
'$ is the probability that two randomly chosen segments across the
fundamental polygon will cross. The difficulty in making this argument
precise, of course, is in quantifying the sense in which the segments
of a typical closed curve look like a random sample of segments. The
arguments below (see sec. \ref{sec:mc}) will make this clear.

The mystery, then, is not why the mean number of self-intersections
grows like $n^{2}$, but rather why the size of typical fluctuations is
of order $n^{3/2}$ and why the limit distribution is Gaussian. This
seems to be connected to geometry.  If the surface $\Sigma $ is
equipped with a finite-area Riemannian metric of negative curvature,
and if the boundary components are (closed) geodesics then
each free homotopy class contains a unique closed geodesic (except for
the free homotopy classes corresponding to the
punctures).  It is
therefore possible to order the free homotopy classes by the length of
the geodesic representative. Fix $L$, and let $\geod_{L}$ be the set
of all free homotopy classes whose closed geodesics are of length
$\leq L$. The main result of \cite{lalley:si2} (see also
\cite{lalley:si1}) describes the variation of the self-intersection
count $N (\alpha)$ as $\alpha$ ranges over the population $\geod_{L}$:

\begin{proclama}{Geometric Sampling Theorem.}
If the Riemannian metric on $\Sigma$ is hyperbolic (i.e.,
constant-curvature -1) then there exists a possibly degenerate
probability distribution $G$ on 
$\zz{R}$ such that for all $a<b$ the proportion of words $\alpha \in
\geod_{L}$ such that 
\[
	a<\frac{N (\alpha)+L^{2}/ (\pi^{2}\chi)}{L}<b
\] 
converges, as $L \rightarrow \infty$, to $G (b)-G (a)$.
\end{proclama}

The limit distribution is not known, but is almost certainly not
Gaussian. The result leaves open the possibility (which we think
unlikely) that the limit distribution is degenerate (that is,
concentrated at a single point); if this were the case, then the true
order of magnitude of the fluctuations might be a fractional power of
$L$.  The Geometric Sampling Theorem
implies that the typical variation in self-intersection count for a
closed geodesic chosen randomly according to hyperbolic length is of
order $L$. Together with the Main Theorem, this suggests
that the much larger variations that occur when sampling by word length
are (in some sense) due to $\sqrt{n}-$variations in hyperbolic length
over the population $\mathcal{F}_{n}$.

The Main Theorem can be reformulated in probabilistic language as
follows (see Appendix~\ref{ssec:probability} for definitions):

\begin{proclama}{Main Theorem$^*$,}  
Let $\Sigma$ be an orientable, compact
surface with boundary and negative Euler characteristic $\chi$, and
let $\kappa$ and $\sigma$ be defined by \eqref{eq:constants}. 
Let $N_{n}$ be the random variable obtained by evaluating the
self-intersection function $N$ at a randomly chosen $\alpha \in
\mathcal{F}_{n}$. Then as $n \rightarrow \infty$,
\begin{equation}\label{eq:main}
	\frac{N_{n}-n^{2}\kappa }{\sigma n^{3/2}} \Longrightarrow \text{Normal} (0,1) 
\end{equation}
where $\text{Normal} (0,1)$ is the standard Gaussian distribution on
$\zz{R}$ and $\Rightarrow$ denotes convergence in distribution.
\end{proclama}

\section{Combinatorics of Self-Intersection Counts}\label{sec:combinatorics}

Our analysis is grounded on a purely combinatorial description of the
self-intersection counts $N(\alpha)$, due to \cite{birman-series:2},
\cite{cohen-lustig}, and \cite{chas:2004}.

Since $\Sigma $ has non-empty boundary, its fundamental group $\pi_{1}
(\Sigma)$ is free. We will work with a generating set of $\pi_{1}
(\Sigma)$ such that each element has a non-self-intersecting
representative (Such a basis is a natural choice to describe
self-intersections of free homotopy classes). Denote by $\gen$ the set
containing the elements of the generating set and their inverses and
by $\ca$ the cardinality of $\gen$.  Thus, $\ca=2-2\chi$, where $\chi$
denotes the Euler characteristic of $\Sigma$. It is not hard to see
that there exists a (non-unique and possibly non-reduced) cyclic word
$\mathcal{O}$ of length $\ca$ such that
\begin{numlist}
\item $\mathcal{O}$ contains each element of $\gen$ exactly once.
\item The surface $\Sigma$ can be obtained as follows: Label the edges
of a polygon with $2 \ca $ sides, alternately (so every other edge is
not labelled) with the letters of $\mathcal{O}$ and glue edges labeled
with the same letter without creating Moebius bands.
\end{numlist} 
This cyclic word $\mathcal{O}$ encodes the intersection and
self-intersection structure of free homotopy classes of curves on
$\Sigma$.

 Since $\pi_{1} (\Sigma)$ is a free group, the elements of $\pi_{1}
(\Sigma)$ can be identified with the \emph{reduced words} (which we
will also call \emph{strings}) in the generators and their inverses.
A string is \emph{joinable} if each cyclic permutation of its letters
is also a string, that is, if its last letter is not the inverse of
its first. A \emph{reduced cyclic word} (also called a \emph{necklace}
) is an equivalence class of joinable strings, where two such strings
are considered equivalent if each is a cyclic permutation of the
other. Denote by $\mathcal{S}_{n},\mathcal{J}_{n}$, and
$\mathcal{F}_{n}$ the sets of strings, joinable strings, and
necklaces, respectively, of length $n$. Since necklaces correspond
bijectively with the conjugacy classes of the fundamental group, the
self-intersection count $\alpha \mapsto N (\alpha)$ can be regarded  as a
function on the set $\mathcal{F}_{n}$ of necklaces. This function
pulls back to a function on the set $\mathcal{J}_{n}$ of joinable
strings, which we again denote by $N (\alpha)$, that is constant on
equivalence classes. By \cite{chas:2004} this function has the form
\begin{equation}\label{eq:SIForm}
	N (\alpha)=\sum_{1\leq i<j\leq n} H (\sigma^{i}\alpha,\sigma^{j}\alpha ),
\end{equation}
where $H=H(\mathcal{O})$ is a symmetric function with values in
$\{0,1\}$ on $\mathcal{J}_{n}\times \mathcal{J}_{n}$ and
$\sigma^{i}\alpha$ denotes the $i$th cyclic permutation of $\alpha $.
(Note: $\sigma^{2}$ also denotes the limiting variance in
\eqref{eq:constants}, but it will be clear from the context which of the
two meanings is in force.)

To describe the function $H$ in the representation \eqref{eq:SIForm},
we must explain the \emph{cyclic ordering} of letters.  For a cyclic
word $\alpha$ (not necessarily reduced), set $o (\alpha)= 1$ if the
letters of $\alpha$ occur in cyclic order in $\mathcal{O}$, set $o
(\alpha)=-1$ if the letters of $\alpha$ occur in \emph{reverse} cyclic
order, and set $o (\alpha)=0$ otherwise.  Consider two (finite or
infinite) strings, $\omega =c_{1}c_{2}\dotsb$ and $\omega
'=d_{1}d_{2}\dotsb$. For each integer $k\geq 2$ define functions
$u_{k}$ and $v_{k}$ of such pairs $(\omega ,\omega ')$ as follows:
First, set $u_{k} (\omega ,\omega ')=0$ unless
\begin{alphlist}
\item  both $\omega$ and $\omega '$ are of length at least $k$; and
\item $c_{1}\not =d_{1}$, $c_{k}\not =d_{k}$, and $c_{j}=d_{j}$ for all $1<j<k$.
\end{alphlist}
For any pair $(\omega ,\omega ')$ such that both (a) and (b) hold, define
\begin{eqnarray*}
u_{k} (\omega ,\omega ') &= & \left\{
        \begin{array}{ll}
1 &\mbox{ if $k = 2$, and $o (\bar{c}_{1}\bar{d}_{1}c_{2}d_{2}) \not =0$;}\\
1 &\mbox{ if $k \ge 3$, and $o(\bar{c}_{1} \bar{d_{1}}c_{2})=o(c_{k}d_{k}\bar{c}_{k-1})$; and} \\
0 & \mbox{otherwise.} 
        \end{array}
\right.
\end{eqnarray*}
Finally, define $v_{2} (\omega ,\omega ')=0$ for all strings $\omega
,\omega '$, and for $k\geq 3$ define $v_{k} (\omega ,\omega ')=0$
unless both $\omega$ and $\omega '$ are of length at least $k$, in
which case
\[
	v_{k} (\omega ,\omega ')= u_{k} (c_{1}c_{2}\dots c_{k},
	\bar{d}_{k}\bar{d}_{k-1}\dots \bar{d}_{1}) .
\]
(Note: The only reason for defining $v_{2}$ is to
avoid having to write separate sums for the functions $v_{j}$ and
$u_{j}$ in formula \eqref{eq:combinatorialFormula} and the arguments
to follow.)
Observe that both $u_{k}$ and $v_{k}$ depend only on the first $k$
letters of their arguments. Furthermore, $u_{k}$ and $v_{k}$  are
defined for arbitrary pairs of strings, finite or infinite; for
\emph{doubly} infinite  sequences $\mathbf{x}=\dotsb
x_{-1}x_{0}x_{1}\dotsb$ we adopt the convention that 
\[
	u_{k} (\mathbf{x})=u_{k} (x_{1}x_{2}\dotsb x_{k})
	\quad \text{and} \quad 
	v_{k} (\mathbf{x})=v_{k} (x_{1}x_{2}\dotsb x_{k}).
\]

\begin{proposition}\label{proposition:chas} \cite{chas:2004}
Let $\alpha$ be a primitive necklace of length $n\geq 2$. Unhook
$\alpha$ at an arbitrary location to obtain a string
$\alpha^{*}=a_{1}a_{2}\dotsb a_{n}$, and let $\sigma^{j}\alpha^{*}$ be
the $j$th cyclic permutation of $\alpha^{*}$. Then
\begin{equation}
\label{eq:combinatorialFormula}
	N (\alpha ) = \sum_{i=1}^{n}\sum_{j=i+1}^{n} \sum_{k=2}^{n} 
	  (u_{k} (\sigma^{i}\alpha^{*},\sigma^{j}\alpha^{*})+v_{k}
	  (\sigma^{i}\alpha^{*},\sigma^{j}\alpha^{*})) .
\end{equation}
\end{proposition}

\section{Proof of the Main Theorem: Strategy}\label{sec:strategy}

Except for the exact values \eqref{eq:constants} of the limiting
constants $\kappa$ and $\sigma^{2}$, which of course depend on the
specific form of the functions $u_{k}$ and $v_{k}$, the conclusions of
the Main Theorem hold more generally for random variables defined by
sums of the form
\begin{equation}\label{eq:generalForm}
	N (\alpha^{*} )= \sum_{i=1}^{n}\sum_{j=i+1}^{n} \sum_{k=2}^{n} 
	  h_{k} (\sigma^{i}\alpha^{*} ,\sigma^{j}\alpha^{*} )
\end{equation}
where $h_{k}$ are real-valued functions on the space of reduced sequences
$\alpha^{*}$ with entries in $\mathcal{G}$ satisfying the hypotheses
(H0)--(H3) below. The function $N$ extends to \emph{necklaces} in an
obvious way: for any necklace $\alpha $ of length $n$, unhook $\alpha$
at an arbitrary place to obtain a joinable string $\alpha^{*}$, then
define $N (\alpha )=N (\alpha^{*})$. Denote by $\lambda _{n}$,
$\mu_{n}$, and $\nu_{n}$ the uniform probability distributions on the
sets $\mathcal{J}_{n}$, $\mathcal{F}_{n}$, and $\mathcal{S}_{n}$,
respectively.

\begin{alphlist}\label{hyperbolic:h}
\item [(H0)] Each function $h_{k}$ is symmetric.
\item [(H1)] There exists $C<\infty$ such that $|h_{k}|\leq C$ for all
$k\geq 1$.
\item [(H2)] For each $k\geq 1$ the function $h_{k}$ depends only on
the first $k$ entries of its arguments.
\item [(H3)] There exist constants $C<\infty$ and $0<\beta  <1$ such
that for all $n\geq k\geq  1$ and $1\leq i<n$,
\[
	E_{\lambda _{n}} |h_{k} (\alpha ,\sigma^{i}\alpha)|
	 \leq
	 C\beta^{k}
\]
\end{alphlist}

\medskip \noindent In view of (H2), the function $h_{k}$ is
well-defined for any pair of sequences, finite or infinite, provided
their lengths are at least $k$.  Hypotheses (H0)--(H2) are clearly
satisfied for $h_{k}=u_{k}+v_{k}$, where $u_{k}$ and $v_{k}$ are as in
formula \eqref{eq:combinatorialFormula} and $u_{1}=v_{1}=0$; see Lemma
\ref{lemma:meanEstimatesJS} of section~\ref{ssec:meanEstimates} for
hypothesis (H3).

\begin{theorem}\label{theorem:general}
Assume that the functions $h_{k}$ satisfy hypotheses (H0)--(H3), and
let $N (\alpha)$ be defined by \eqref{eq:generalForm} for all
necklaces $\alpha$ of length $n$.Then there exist constants $\kappa$
and $\sigma^{2}$ (given by equations \eqref{eq:constantsMostGeneral}
below) such that if $F_{n}$ is the distribution of the random variable $ ( N
(\alpha) -n^{2}\kappa )/ n^{3/2}$ under the probability measure
$\mu_{n}$.  Then for certain constants $\kappa \in \zz{R}$ and $\sigma
\geq 0$,
\begin{equation}\label{eq:normalConvergenceGeneral}
	F_{n} \Longrightarrow \text{Normal} (0,\sigma^{2}).
\end{equation}
\end{theorem}

\medskip

Formulas for the limiting constants $\kappa ,\sigma $ are given (in
more general form) in Theorem~\ref{theorem:U-CLT} below. In
section~\ref{sec:variance} we will show that in the case of particular
interest, namely  $h_{k}=u_{k}+v_{k}$ where $u_{k},v_{k}$ are as in
Proposition~\ref{proposition:chas}, the constants $\kappa $ and
$\sigma $ defined in Theorem~\ref{theorem:U-CLT} assume the values
\eqref{eq:constants} given in the statement of the Main Theorem.

Modulo the proof of Lemma \ref{lemma:meanEstimatesJS} and the
calculation of the constants $\sigma $ and $\kappa $, the Main Theorem
follows directly from Theorem~\ref{theorem:general}.  The proof of
Theorem~\ref{theorem:general} will proceed roughly as follows. First
we will prove (Lemma~\ref{lemma:MC}) that there is a
shift-invariant, Markov probability measure $\nu $ on the space
$\mathcal{S}_{\infty}$ of infinite sequences
$\mathbf{x}=x_{1}x_{2}\dotsb$ whose marginals (that is, the
push-forwards under the projection mappings to $\mathcal{S}_{n}$) are
the uniform distributions $\nu_{n}$.  Using this representation we
will prove, in subsection~\ref{ssec:j-to-s}, that when $n$ is large the
distribution of $N (\alpha)$ under $\mu_{n}$ differs negligibly from
the distribution of a related random variable defined on the Markov
chain with distribution $\nu$. See Proposition~\ref{proposition:close}
for a precise statement.  Theorem~\ref{theorem:general} will then
follow from a general limit theorem for certain \emph{U-statistics} of
Markov chains (see Theorem~\ref{theorem:U-CLT}).

\section{The Associated Markov Chain}\label{sec:mc}

\subsection{Necklaces, strings, and joinable strings}\label{ssec:counts}

Recall that a \emph{string} is a sequence with entries in the set
$\mathcal{G}$ of generators and their inverses such that no two
adjacent entries are inverses. A finite string is \emph{joinable} if
its first and last entries are not inverses. The sets of length-$n$
strings, joinable strings, and necklaces are denoted by
$\mathcal{S}_{n},\mathcal{J}_{n}$, and $\mathcal{F}_{n}$,
respectively, and the uniform distributions on these sets are denoted
by $\nu_{n}, \lambda_{n}$, and $\mu_{n}$. Let $A$ be the involutive
permutation matrix with rows and columns indexed by $\mathcal{G}$
whose entries $a (x,y)$ are $1$ if $x$ and $y$ are inverses and $0$
otherwise. Let $B$ be the matrix with all entries $1$. Then for any
$n\geq 1$,
\[
	|\mathcal{S}_{n}|=\mathbf{1}^{T} (B-A)^{n-1}\mathbf{1}
	\quad \text{and} \quad 
	|\mathcal{J}_{n}| = \text{trace} (B-A)^{n-1},
\]
where $\mathbf{1}$ denotes the (column) vector all of where entries
are $1$.  Similar formulas can be written for the number of strings
(or joinable strings) with specified first and/or last entry.  The
matrix $B-A$ is a Perron-Frobenius matrix with lead eigenvalue
$(g-1)$; this eigenvalue is simple, so both $|\mathcal{S}_{n}|$ and
$|\mathcal{J}_{n}|$ grow at the precise exponential rate $(g-1)$, that
is, there exist positive constants $C_{\mathcal{S}}=g/ (g-1)$ and
$C_{\mathcal{J}}$ such that
\[
	|\mathcal{S}_{n}|\sim C_{\mathcal{S}} (g-1)^{n}	
	\quad \text{and} \quad 
	|\mathcal{J}_{n}| \sim C_{\mathcal{J}} (g-1)^{n}.	
\]

Every necklace of length $n$ can be obtained by joining the ends of a
joinable string, so there is a natural surjective mapping
$p_{n}:\mathcal{J}_{n} \rightarrow \mathcal{F}_{n}$. This mapping is
nearly $n$ to $1$: In particular, no necklace has more than $n$
pre-images, and the only necklaces that do not have exactly $n$
pre-images are those which are periodic with some period $d|n$ smaller
than $n$. The number of these exceptional necklaces is vanishingly
small compared to the total number of necklaces. To see this, observe
that the total number of \emph{strings} of length $n\geq 2$ is $ \ca (
\ca -1)^{n-1}$; hence, the number of \emph{joinable} strings is
between $ \ca ( \ca -1)^{n-2}$ and $ \ca ( \ca -1)^{n-1}$.  The number
of length-$n$ strings with period $<n$ is bounded above by
\[
	\sum_{d|n}  \ca  ( \ca -1)^{d-1}\leq
	\text{constant}\times( \ca -1)^{n/2}. 
\]
This is of smaller exponential order of magnitude than
$|\mathcal{J}_{n}|$, so for large $n$ most necklaces of length $n$
will have exactly $n$ pre-images under the projection $p_{n}$.
Consequently, as $n \rightarrow \infty $
\[
	|\mathcal{F}_{n}|\sim C_{\mathcal{J}} (g-1)^{n}/n.
\]
More important, this implies the following.

\begin{lemma}\label{lemma:tv-necklaces-JS}
Let $\lambda_{n}$ be the uniform probability distribution on the set
$\mathcal{J}_{n}$, and let $\mu_{n}\circ p_{n}^{-1}$ be the
push-forward to $\mathcal{J}_{n}$ of the uniform distribution on
$\mathcal{F}_{n}$. Then
\begin{equation}\label{eq:tv-necklaces-JS}
	\lim_{n \rightarrow \infty} \xnorm{\lambda_{n}-\mu_{n}\circ
	p_{n^{-1}}}_{TV} =0.
\end{equation}
\end{lemma}

Here $\xnorm{\cdot}_{TV}$ denotes the total variation norm on measures
-- see the Appendix. By Lemma~\ref{lemma:tools} of the Appendix, it
follows that the distributions of the random variable $N (\alpha )$
under the probability measures $\lambda_{n}$ and $\mu_{n}$ are
asymptotically indistinguishable.

\subsection{The associated Markov measure}\label{ssec:markovMeasure}

The matrix $(B-A)$ has the convenient feature that its row sums and
column sums are all $g-1$. Therefore, the matrix $\zz{P}:= (B-A)/
(g-1)$ is a stochastic matrix, with entries
\begin{eqnarray} \label{eq:tps}
p (a,b) &= & \left\{
        \begin{array}{ll}
\theta &\mbox{ if $b\not =a^{-1}$, and}\\
0 & \mbox{otherwise, where} 
        \end{array}
\right. \\
\label{eq:thetaDef}
	\theta&=&(\ca -1)^{-1}.
\end{eqnarray}
In fact, $\zz{P}$ is \emph{doubly stochastic}, that is, both its rows
and columns sum to $1$.  Moreover, $\zz{P}$ is aperiodic and
irreducible, that is, for some $k\geq 1$ (in this case $k=2$) the
entries of $\zz{P}^{k}$ are strictly positive. It is an elementary
result of probability theory that for any aperiodic, irreducible,
doubly stochastic matrix $\zz{P}$ on a finite set $\mathcal{G}$ there
exists a shift-invariant probability measure $\nu$ on sequence space
$\mathcal{S}_{\infty}$, called a \emph{Markov measure}, whose value on
the cylinder set $C (x_{1}x_{2}\dotsb x_{n})$ consisting of all
sequences whose first $n$ entries are $x_{1}x_{2}\dotsb x_{n}$ is
\begin{equation}\label{eq:markovMeasure}
	\nu (C(x_{0}x_{1}x_{2}\dotsb x_{n}))= \frac{1}{|\mathcal{G}|}
	\prod_{i=1}^{n-1}p (x_{i},x_{i+1})
\end{equation}
Any random sequence $\mathbf{X}= (X_{1}X_{2}\dotsb)$
valued in $\mathcal{S}_{\infty}$, defined on any probability space
$(\Omega, P)$, whose distribution is $\nu $ is called a 
\emph{stationary Markov chain with transition probability matrix}
$\zz{P}$. In particular, the coordinate process on
$(\mathcal{S}_{\infty }, \nu)$ is a Markov chain with t.p.m. $\zz{P}$.

\begin{lemma}\label{lemma:MC}
Let $\mathbf{X}= (X_{1}X_{2}\dotsc )$ be a stationary
Markov chain with transition probability matrix $\zz{P}$ defined by
\eqref{eq:tps}.  Then for any $n\geq 1$ the distribution of the random
string $X_{1}X_{2}\dotsb X_{n}$ is the uniform distribution $\nu_{n}$
on the set $\mathcal{S}_{n}$.
\end{lemma}

\begin{proof}
The transition probabilities $p (a,b)$ take only two values, $0$ and
$\theta $, so for any $n$ the nonzero cylinder probabilities
\eqref{eq:markovMeasure} are all the same. Hence, the distribution of
$X_{1}X_{2}\dotsb X_{n}$ is the uniform distribution on the set of all
strings $\xi =x_{1}x_{2}\dotsb x_{n}$ such that the cylinder
probability $\nu (C (\xi ))$ is positive. These are precisely the
\emph{strings} of length $n$. 
\end{proof}

\subsection{Mixing properties of the Markov
chain}\label{ssec:mixingRate} Because the transition probability
matrix $\zz{P}$ defined by \eqref{eq:tps} is aperiodic and irreducible,
the $m-$step transition probabilities (the entries of the $m$th power
$\zz{P}^{m}$ of $\zz{P}$) approach the stationary (uniform)
distribution exponentially fast. The one-step transition probabilities
\eqref{eq:tps} are simple enough that precise bounds can be given:

\begin{lemma}\label{lemma:exponentialErgodicity} 
The $m-$step transition probabilities $p_{m} (a,b)$ of the Markov
chain with $1-$step transition probabilities \eqref{eq:tps} satisfy
\begin{equation}\label{eq:exponentialErgodicity}
	\left| p_{m} (a,b)-\frac{1}{ \ca }\right|\leq \theta^{m}
\end{equation}
where $\theta =1/ ( \ca -1)$.
\end{lemma}

\begin{proof}
Recall that $\zz{P}=\theta
(B-A)$ where $B$ is the matrix with all entries $1$ and $A$ is
an involutive permutation matrix.  Hence,
$BA=AB=B$ and $B^{2}= \ca B =((\theta+1)/\theta)B$. This implies, by a routine induction
argument, that for every integer $m\geq 1$,
\begin{align*}
	(B-A)^{m} &= \left(\frac{\theta^{-m}+1}{ \ca } \right)B-A \quad
	\text{if $m$ is odd, and}\\
	(B-A)^{m} &= \left(\frac{\theta^{-m}-1}{ \ca } \right)B +I \quad
	\text{if $m$ is even.}
\end{align*}
The inequality \eqref{eq:exponentialErgodicity} follows directly.
\end{proof}

Following is a useful way to reformulate the exponential convergence
\eqref{eq:exponentialErgodicity}. Let $\mathbf{X}= (X_{j})_{j\in \zz{Z}}$ be 
a stationary Markov chain with transition probabilities \eqref{eq:tps}.
For any finite subset $J\subset \zz{N}$, let $X_{J}$ denote the
restriction of $\mathbf{X}$ to the index set $J$, that is,
\[
	X_{J}= (X_{j})_{j\in J};
\]
for example, if $J$ is the interval $[n]$ then
$X_{J}$ is just the random string $X_{1}X_{2}\dotsb X_{n}$. Denote
by $\nu_{J}$ the distribution of $X_{J}$, viewed as a probability
measure on the set $\gen^{J}$; thus, for any subset
$F\subset \gen^{J}$,
\begin{equation}\label{eq:muJ}
	\nu_{J} (F)=P\{X_{J}\in F \}.
\end{equation}
If $J,K$ are non-overlapping subsets of $\zz{N}$, then $\nu_{J\cup K}$
and $\nu_{J}\times \nu_{K}$ are both probability measures on
$\gen^{J\cup K}$, both with  support set equal to the set of
all restrictions of infinite strings.

\begin{lemma}\label{lemma:expSeparation}
Let $J,K\subset \zz{N}$ be two finite subsets such that 
$\max (J)+m\leq \min (K)$ for some $m\geq 1$. Then
on their common support set,
\begin{equation}\label{eq:expSeparation}
		\frac{1- \ca  \theta^{m-1}}{1+ \ca  \theta^{m-1}}
	\leq \frac{d\nu_{J\cup K}}{d\nu_{J}\times \nu_{K}}
	 \leq  \frac{1+ \ca  \theta^{m-1}}{1- \ca  \theta^{m-1}}
\end{equation}
where $\theta =1/ ( \ca -1)$ and $d\alpha /d\beta$ denotes the
Radon-Nikodym derivative (``likelihood ratio'') of the probability
measure $\alpha$ and $\beta$.
\end{lemma}

\begin{proof}
It suffices to consider the special case where $J$ and $K$ are
intervals, because the general case can be deduced by summing over
excluded variables.  Furthermore, because the Markov chain is
stationary, the measures $\nu_{J}$ are invariant by translations (that
is, $\nu_{J+n}=\nu_{J}$ for any $n\geq 1$), so we may assume that
$J=[1,n]$ and $K=[n+m, n+q]$. Let $x_{J\cup K}$ be the restriction of some
infinite string to $J\cup K$; then
\begin{align*}
	\nu_{J}\times \nu_{K} (x_{J\cup K}) &=\pi
	(x_{1})\left(\prod_{j=1}^{n-1}p (x_{j},x_{j+1}) \right) 
	\pi (x_{n+m}) \left(\prod_{j=n+m}^{m+q-1}p (x_{j}, x_{j+1})
	\right)
	\quad \text{and}\\
	\nu_{J\cup K} (x_{J\cup K})&=\pi
	(x_{1})\left(\prod_{j=1}^{n-1}p (x_{j},x_{j+1}) \right) p_{m} 
	(x_{n},x_{n+m}) \left(\prod_{j=n+m}^{m+q-1}p (x_{j}, x_{j+1})
	\right) .	
\end{align*}
The result now follows directly from the double inequality
\eqref{eq:exponentialErgodicity}. 
\end{proof}

\subsection{From random joinable strings to random
strings}\label{ssec:j-to-s} Since $\mathcal{J}_{n}\subset
\mathcal{S}_{n}$, the uniform distribution $\lambda_{n}$ on
$\mathcal{J}_{n}$ is gotten by restricting the uniform distribution
$\nu_{n}$ on  $\mathcal{S}_{n}$ to $\mathcal{J}_{n}$ and then
renormalizing:
\[
	\lambda_{n} (F)=\frac{\nu_{n} (F\cap \mathcal{J}_{n})}{\nu_{n}
	(\mathcal{J}_{n})}.
\]
Equivalently, the distribution of a random \emph{joinable} string is
the conditional distribution of a random string given that its first
and last entries are not inverses. Our goal here is to show that the
distributions of the random variable $N (\alpha)$ defined by
\eqref{eq:generalForm}  under the probability measures $\lambda_{n}$
and $\nu_{n}$ differ negligibly when $n$ is large. For this we will
show first that the distributions under $\lambda_{n}$ and $\nu_{n}$,
respectively, of the \emph{substring} gotten by deleting
the last $n^{1/2-\varepsilon}$ letters are close in total variation distance;
then we will show that changing the last $n^{1/2-\varepsilon}$ letters
has only a small effect on the value of $N (\alpha )$.

\begin{lemma}\label{lemma:tv}
Let $X_{1}X_{2}\dotsb X_{n}$ be a random string of length $n$, and
$Y_{1}Y_{2}\dotsb Y_{n}$ a random \emph{joinable} string.  For any
integer $m\in [1,n-1]$ let $\nu_{n,m}$ and $\lambda_{n,m}$ denote the
distributions of the random substrings $X_{1}X_{2}\dotsb
X_{n-m}$ and $Y_{1}Y_{2}\dotsb Y_{n-m}$. Then $\lambda_{n,m} \ll
\nu_{n,m}$, and  the Radon-Nikodym
derivatives satisfy
\begin{equation}\label{eq:ac}
	\frac{1- \ca  \theta^{m}}{1+ \ca  \theta^{m}}
	\leq \frac{d\lambda_{n,m}}{d\nu_{n,m}}
	 \leq  \frac{1+ \ca  \theta^{m}}{1- \ca  \theta^{m}}
\end{equation}
where $\theta =1/( \ca -1) $.  Consequently, 
\begin{equation}\label{eq:tv}
	\xnorm{\nu_{n,m}-\lambda_{n,m}}_{TV} \leq  2\left( \frac{1+ \ca 
	\theta^{m}}{1- \ca  \theta^{m}} -1 \right).
\end{equation}
\end{lemma}

\begin{proof}
Consider first the case $m=0$. Recall that there are precisely
$ \ca  ( \ca -1)^{n-1}$ strings of length $n$, and at
least $ \ca  ( \ca -1)^{n-2} ( \ca -2)$ of
these are joinable. Hence, the likelihood ratio
$d\lambda_{n,0}/d\nu_{n,0}$ is bounded above by
$( \ca - 1)/( \ca -2)$ and below by $1$. This proves
\eqref{eq:ac} for $m=0$. The case $m=1$ follows similarly.

The general case $m\geq 2$ follows from the exponential ergodicity
estimates \eqref{eq:exponentialErgodicity} by an argument much like
that used to prove Lemma~\ref{lemma:expSeparation}. For any string
$x_{1}x_{2}\dotsb x_{n-m}$ with initial letter $x_{1}=a$,
\[
	\nu_{n,m} (x_{1}x_{2}\dotsb x_{n-m})=
	\frac{1}{ \ca }\prod_{i=1}^{n-m-1}p (x_{i},x_{i+1}).
\]
Similarly, by Lemma~\ref{lemma:MC}, 
\[
	\lambda_{n,m}  (x_{1}x_{2}\dotsb x_{n-m})=
	\frac{1}{ \ca } \left( \prod_{i=1}^{n-m-1}p
	(x_{i},x_{i+1})\right)
	\frac{\sum_{b\not =x_{1}^{-1}}p_{k}
	(x_{n-m},b)}{ \ca ^{-1}\sum_{a}\sum_{b\not =a^{-1}}p_{n} (a,b)} .
\]
Inequality \eqref{eq:exponentialErgodicity} implies that the last
fraction in this expression is between the bounding fractions in
\eqref{eq:ac}. The bound on the total variation distance
between the two measures follows routinely.
\end{proof}

\begin{corollary}\label{corollary:meanEstB}
Let $\mathbf{X}$ be a stationary Markov chain with transition
probability matrix $\zz{P}$.
Assume that the functions $h_{k}$ satisfy hypotheses (H0)--(H3) of
section~\ref{sec:strategy}. Then for all $k,i\geq 1$,
\begin{equation}\label{eq:meanEstB}
	E|h_{k} (\mathbf{X},\tau^{i}\mathbf{X})|\leq C\beta^{k}
\end{equation}
\end{corollary}

\begin{proof}
The function  $h_{k} (\mathbf{x},\tau^{i}\mathbf{x})$ is a
function only of the coordinates $x_{1}x_{2}\dotsb x_{i+k}$, and so
for any joinable string $\mathbf{X}$ of length $>i+k$,
\[
	h_{k} (\mathbf{x},\tau^{i}\mathbf{x})=h_{k} (\mathbf{x},\sigma
	^{i}\mathbf{x}).
\]
By Lemma \ref{lemma:tv}, the difference in total
variation norm between the distributions of the substring $x_{1}x_{2}\dotsb x_{i+k}$
under the measures $\lambda_{n}$ and $\nu_{n}$ converges to $0$ as $n
\rightarrow \infty$. Therefore,
\[
	E|h_{k} (\mathbf{X},\tau^{i}\mathbf{X})|=\lim_{ n \rightarrow
	\infty} E_{\lambda_{n}}|h_{k} (\alpha ,\sigma^{i}\alpha )|\leq C\beta^{k}.
\]
\end{proof}

Now we are in a position to compare the distribution of the  random
variable $N (\alpha )$ under $\mu_{n}$ with the distribution of a
corresponding random variable $N^{S}_{n}$ on the sequence space
$\mathcal{S}_{\infty }$ under the measure $\nu $. The function $N^{S}_{n}$
is defined by
\begin{equation}\label{eq:NSdef}
	N^{S}_{n} (\mathbf{x})=\sum_{i=1}^{n}\sum_{j=i+1}^{n}
	\sum_{k=1}^{\infty} h_{k}
	(\tau^{i}\mathbf{x},\tau^{j}\mathbf{x}). 
\end{equation}

\begin{proposition}\label{proposition:close}
Assume that the functions $h_{k}$ satisfy hypotheses (H0)--(H3), and
let $\kappa =\sum_{k=1}^{\infty }EH_{k} (\mathbf{X})$.
Let $F_{n}$ be the distribution of the random variable
$(N(\alpha)-n^{2}\kappa )/n^{3/2}$ under the uniform probability
measure $\mu_{n}$ on $\mathcal{F}_{n}$, and $G_{n}$ the distribution
of $( N^{S}_{n} (\mathbf{x})-\kappa n^{2})/n^{3/2}$ under $\nu$. Then
for any metric $\varrho$ that induces the topology of weak convergence
on probability measures,
\begin{equation}\label{eq:close}
	\lim_{n \rightarrow \infty} \varrho (F_{n},G_{n})=0.
\end{equation}
Consequently, $F_{n}\Rightarrow \Phi_{\sigma }$ if and only if
$G_{n}\Rightarrow \Phi_{\sigma }$. 
\end{proposition}

\begin{proof}
Let $F_{n}'$ be the distribution of the random variable
$(N(\alpha)-n^{2}\kappa )/n^{3/2}$ under the uniform probability
measure $\lambda_{n}$ on $\mathcal{J}_{n}$. By
Lemma~\ref{lemma:tv-necklaces-JS}, the total variation distance
between $\lambda_{n}$ and $\mu_{n}\circ p_{n}^{-1}$ is vanishingly
small for large $n$. Hence, by Lemma~\ref{lemma:tools} and the fact
that total variable distance is never increased by mapping
(cf. inequality \eqref{eq:mappingPrinciple} of the Appendix),
\[
	\lim_{n \rightarrow \infty }\varrho (F_{n},F_{n}')=0.
\]
Therefore, it suffices to prove  \eqref{eq:close} with $F_{n}$
replaced by $F_{n}'$. 

Partition the sums \eqref{eq:generalForm} and \eqref{eq:NSdef} as
follows. Fix $0<\delta <1/2$  and  set $m=m (n)=[n^{\delta}]$. By
hypothesis (H3) and Corollary  \eqref{corollary:meanEstB}, 
\begin{align*}
	E_{\mu }\sum_{i=1}^{n}\sum_{j=i+1}^{n}\sum_{k>m (n)} |h_{k}
	(\tau^{i}\mathbf{x},\tau^{j}\mathbf{x})| &\leq C n^{2}\beta^{m (n)}
	\quad \text{and}\\ 
	E_{\lambda_{n}}\sum_{i=1}^{n}\sum_{j=i+1}^{n}\sum_{k>m
	(n)}|h_{k} (\sigma^{i}\alpha ,\sigma^{j}\alpha )|  &\leq C
	n^{2}\beta^{m (n)} .
\end{align*}
These upper bounds are rapidly decreasing in $n$.  Hence, by Markov's
inequality (i.e., the crude bound $P\{|Y|>\varepsilon \}\leq
E|Y|/\varepsilon$), the distributions of both of the sums converge
weakly to $0$ as $n \rightarrow \infty$. Thus, by
Lemma~\ref{lemma:tools}, to prove the proposition it suffices to prove
that
\[
	\lim_{n \rightarrow \infty }\varrho (F_{n}^{A},G_{n}^{A})=0
\]
where $F_{n}^{A}$ and $G_{n}^{A}$ are the distributions of the
truncated sums obtained by replacing the inner sums in
\eqref{eq:generalForm} and \eqref{eq:NSdef} by the sums over $1\leq
k\leq m (n)$. 

The outer sums in \eqref{eq:generalForm} and \eqref{eq:NSdef} are over
pairs of indices $1\leq i<j\leq n$. Consider those pairs for which
$j>n-2m (n)$: there are only $2nm (n)$ of these. Since $nm (n)=O
(n^{1+\delta })$ and $\delta <1/2$, and since each term in
\eqref{eq:generalForm} and \eqref{eq:NSdef} is bounded in absolute
value by a constant $C$ (by Hypothesis (H1)), the sum over those index
pairs $i<j$ with $n-2m (n)<j\leq n$ is $o (n^{3/2})$. Hence, by
Lemma~\ref{lemma:tools}, it suffices to prove that 
\[
	\lim_{n \rightarrow \infty }\varrho (F_{n}^{B},G_{n}^{B})=0
\]
where $F_{n}^{B}$ and $G_{n}^{B}$ are the distributions under
$\lambda_{n}$ and $\nu $ of the sums \eqref{eq:generalForm} and
\eqref{eq:NSdef} with the limits of summation changed to
$i<j<n-2m (n)$ and $k\leq m (n)$. Now if $i<j<n-2m (n)$ and $k\leq m
(n)$ then $h_{k} (\tau^{i}\mathbf{x},\tau^{j}\mathbf{x})$ and $h_{k}
(\sigma^{i}\alpha ,\sigma^{j}\alpha)$ depend only on the first
$n-n(m)$ entries of $\mathbf{x}$ and $\alpha $. Consequently, the
distributions $F_{n}^{B}$ and $G_{n}^{B}$ are the distributions of the
sums
\[
	\sum_{i=1}^{n-2m (n)} \sum_{j=i+1}^{n-2m (n)}\sum_{k\leq m (n)} h_{k}
	(\tau^{i}\mathbf{x},\tau^{j}\mathbf{x}) 
\]
under the probability measures $\lambda_{n,m}$ and $\nu_{n,m}$,
respectively, where  $\lambda_{n,m}$ and $\nu_{n,m}$ are as defined in
Lemma~\ref{lemma:tv}. But the total variation distance between
$\lambda_{n,m}$ and $\nu_{n,m}$ converges to zero, by
Lemma~\ref{lemma:tv}. Therefore, by the mapping principle
\eqref{eq:mappingPrinciple} and Lemma~\ref{lemma:tools}, 
\[
	\varrho (F_{n}^{B},G_{n}^{B}) \longrightarrow 0.
\]

\end{proof}

\subsection{Mean Estimates}\label{ssec:meanEstimates}

In this section we show that  the hypothesis (H3) is satisfied by the
functions $h_{k}=u_{k}+v_{k}$, where $u_{k}$ and $v_{k}$ are as in
Proposition~\ref{proposition:chas}.

\begin{lemma}\label{lemma:meanEstimatesJS}
Let $\sigma^{i} $ be the $i$th cyclic shift on the set
$\mathcal{J}_{n}$ of joinable sequences $\alpha $. There exists
$C<\infty$ such that  for all $2\leq k\leq n$
and $0\leq i<j<n$, 
\begin{equation}\label{eq:meanEstimatesJS}
	\begin{array}{rcl}
E_{\lambda_{n}}u_{k} (\sigma ^{i}\beta ,\sigma ^{j}\beta)&\leq &C\theta^{k/2}\quad \text{and}\\ 
E_{\lambda_{n}}v_{k} (\sigma ^{i}\beta ,\sigma ^{j}\beta)&\leq &C\theta^{k/2}.
\end{array}
\end{equation}
\end{lemma}

\begin{proof}
Because the measure $\lambda_{n}$ is invariant under both cyclic
shifts and the reversal function, it suffices to prove the estimates
only for the case where one of the indices $i,j$ is $0$. If the proper
choice is made ($i=0$ and $j\leq n/2$), then a necessary condition for
$u_{k} (\alpha ,\sigma^{j}\alpha )\not =0$ is that the strings $\alpha
$ and $\sigma^{j}\alpha $ agree in their second through their
$(k-1)/2$th slots. By routine counting arguments (as in
section~\ref{ssec:counts}) it can be shown that the number of joinable
strings of length $n$ with this property is bounded above by $C
(g-1)^{n-k/2}$, where $C<\infty$ is a constant independent of both $n$
and $k\leq n$. This proves the first inequality. A similar argument
proves the second.
\end{proof}

\section{U-Statistics of Markov chains}\label{sec:u-statistics}

Proposition~\ref{proposition:close} implies that for large $n$ the
distribution $F_{n}$ considered in Theorem~\ref{theorem:general} is
close in the weak topology to the distribution $G_{n}$ of the random
variable $N^{S}_{n}$ defined by \eqref{eq:NSdef} under the Markov
measure $\nu$. Consequently, if it can be shown that $G_{n}\Rightarrow
\Phi_{\sigma}$ then the conclusion $F_{n}\Longrightarrow
\Phi_{\sigma}$ will follow, by Lemma~\ref{lemma:tools} of the
Appendix. This will prove Theorem~\ref{theorem:general}.

Random variables of the form \eqref{eq:NSdef} are known generically in
probability theory as \emph{U-statistics} (see
\cite{hoeffding}). Second order $U-$statistics of Markov chains are
defined as follows.  Let $\mathbf{Z}=Z_{1}Z_{2}\dotsb $ be a
stationary, aperiodic, irreducible Markov chain on a finite state
space $\mathcal{A}$ with transition probability matrix $\zz{Q}$ and
stationary distribution $\pi$.  Let $\tau $ be the forward shift on
sequence space $\mathcal{A}^{\zz{N}}$. The $U-$statistics of order $2$
with kernel $h:\mathcal{A}^{\zz{N}}\times
\mathcal{A}^{\zz{N}}\rightarrow \zz{R}$ are the random variables
\[
	W_{n}=\sum_{i=1}^{n}\sum_{j=i+1}^{n} h
	(\tau^{i}\mathbf{Z},\tau^{j}\mathbf{Z}) .
\]
The Hoeffding projection of a kernel $h$ is the function
$H:\mathcal{A}^{\zz{N}} \rightarrow \zz{R}$ defined by
\[
	H (\mathbf{z})=Eh (\mathbf{z},\mathbf{Z}).
\]

\begin{theorem}\label{theorem:U-CLT}
Suppose that $h=\sum_{k=1}^{\infty }h_{k}$ where $\{h_{k} \}_{k\geq
1}$ is a sequence of kernels satisfying hypotheses (H0)--(H2) and the
following: There exist constants $C<\infty $ and $0<\beta <1$ such
that for all $k,i\geq 1$,
\begin{equation}\label{eq:hypExpDecay}
	E|h_{k} (\mathbf{Z},\tau^{i}\mathbf{Z})|\leq C\beta^{k}.
\end{equation}
Then as $n \rightarrow \infty$,
\begin{equation}\label{eq:target}
	\frac{W_{n}-n^{2}\kappa}{n^{3/2}}\Longrightarrow \Phi_{\sigma }
\end{equation}
where the constants $\kappa $ and $\sigma^{2}$ are
\begin{equation}\label{eq:constantsMostGeneral}
		\kappa = \sum_{k=1}^{\infty } EH_{k} (\mathbf{Z})
\quad \text{and} \quad \sigma^{2}=\lim_{n \rightarrow \infty}
\frac{1}{n}E\left(\sum_{i=1}^{n}\sum_{k=1}^{\infty}H_{k} ( \tau^{i}
\mathbf{Z}) -n\kappa \right)^{2} .
\end{equation}
 \end{theorem}

There are similar theorems in the literature, but all require some
degree of additional continuity of the kernel $h$.  In the special
case where all but finitely many of the functions $h_{k}$ are
identically $0$ the result is a special case of Theorem~1 of
\cite{denker-keller} or Theorem~2 of \cite{iosifescu}.  If the
functions $h_{k}$ satisfy the stronger hypothesis that $|h_{k}|\leq
C\beta^{k}$ pointwise then the result follows (with some work) from
Theorem~2 of \cite{iosifescu}. Unfortunately, the special case of
interest to us, where $h_{k}=u_{k}+v_{k}$ and $u_{k},v_{k}$ are the
functions defined in sec.~\ref{sec:combinatorics}, does not satisfy
this hypothesis.

The rest of section~\ref{sec:u-statistics} is devoted to the
proof. The main step is to reduce the problem to the special case where
all but finitely many of the functions $h_{k}$ are identically $0$ by
approximation; this is where the hypothesis \eqref{eq:hypExpDecay}
will be used. The special case, as already noted, can be deduced from
the results of \cite{denker-keller} or \cite{iosifescu}, but
instead we shall give a short and elementary argument.

In proving Theorem~\ref{theorem:U-CLT} we can assume that all of the
Hoeffding projections $H_{k}$ have mean 
\[
	 EH_{k} (\mathbf{Z})=0,
\]
because subtracting a constant from both $h$ and $\kappa $ has no
effect on the validity of the theorem. Note that this does \emph{not}
imply that $Eh_{k} (\tau^{i}\mathbf{Z},\tau^{j}\mathbf{Z})=0$, but it
does imply (by Fubini's theorem) that if $\mathbf{Z}$ and
$\mathbf{Z}'$ are independent copies of the Markov chain then
\[
	Eh_{k} (\mathbf{Z},\mathbf{Z}')=0.
\]

\subsection{Proof in the special case}\label{ssec:special}
If all but finitely many of the functions $h_{k}$ are $0$ then for
some  finite value of $K$  the kernel $h$ depends only on the first
$K$ entries of its arguments. 

\begin{lemma}\label{lemma:K1}
Without loss of generality, we can assume that $K=1$.
\end{lemma}

\begin{proof}
If $Z_{1}Z_{2}\dotsb $ is a stationary Markov chain, then
so is the sequence $Z_{1}^{K}Z_{2}^{K}\dotsb $ where
\[
	Z_{i}^{K}=Z_{i}Z_{i+1}\dotsb Z_{i+K}
\]
is the length-$(K+1)$ word obtained by concatenating the $K+1$ states of
the original Markov chain following $Z_{i}$. Hence, the $U-$statistics
$W_{n}$ can be represented as $U-$statistics on a different Markov
chain with kernel depending only on the first entries of its
arguments.  It is routine to check that the constants $\kappa $ and
$\sigma^{2}$ defined by \eqref{eq:constantsMostGeneral} for the chain
$Z^{K}_{n}$ equal those defined by \eqref{eq:constantsMostGeneral} for
the original chain.
\end{proof}

Assume now that $h$ depends only on the  first entries of its
arguments. Then the Hoeffding projection $H$ also depends only on the
first entry of its argument, and can be written as
\[
	H (z)=Eh (z,Z_{1})=\sum_{z'\in \mathcal{A}} h (z,z') \pi (z').
\]
Since the Markov chain $Z_{n}$
stationary and ergodic, the covariances $EH (Z_{i})H (Z_{i+n})=EH
(Z_{1})H (Z_{1+n})$ decay exponentially in $n$, so the limit
\begin{equation}\label{eq:limitVariance}
	\sigma^{2}:=\lim_{n \rightarrow
	\infty}\frac{1}{n}E\left(\sum_{j=1}^{n}H (Z_{j}) \right)^{2} 
\end{equation}
exists and is nonnegative. It is an elementary fact that
$\sigma^{2}>0$ unless $H\equiv 0$.  Say that the kernel $h$ is
\emph{centered} if this is the case. If $h$
is not centered then the adjusted kernel
\begin{equation}\label{eq:kernelAdjusted}
	h^{*} (z,z'):=h (z,z') -H (z)-H (z')
\end{equation}
is centered, because its Hoeffding projection satisfies
\begin{align*}
	H^{*} (z):&=Eh^{*} (z,Z_{1}) \\
\notag    &=Eh (z,Z_{1}) -EH (z)-EH (Z_{1})\\
\notag    &=H (z)-H (z) -0.
\end{align*}

Define
\[
	T_{n}=\sum_{i=1}^{n}\sum_{j=1}^{n} h (Z_{i},Z_{j})
	\quad \text{and} \quad 
	D_{n}=\sum_{i=1}^{n} h (Z_{i},Z_{i});
\]
then since the kernel $h$ is symmetric,
\begin{equation}\label{eq:wts}
	W_{n}=\frac{1}{2} (T_{n}-D_{n}).
\end{equation}

\begin{proposition}\label{proposition:centeredKernel}
If $h$ is centered, then 
\begin{equation}\label{eq:centeredConvergence}
	T_{n}/n \Longrightarrow  Q
\end{equation}
where $Q$ is a quadratic form in no more than $m=|\mathcal{Z}|$ independent,
standard normal random variables.
\end{proposition}

\begin{proof}
Consider the linear operator $L_{h}$ on $\ell^{2} (\mathcal{Z},\pi)$
defined by
\[
	L_{h}f (z):=\sum_{z'\in \mathcal{Z}} h (z,z')f (z')\pi (z').
\]
This operator is symmetric (real Hermitean), and consequently has a
complete set of orthonormal real eigenvectors $\varphi_{j} (z)$ with
real eigenvalues $\lambda_{j}$. Since $h$ is centered, the
constant function $\varphi_{1}:=1/\sqrt{m}$ is an eigenvector with
eigenvalue $\lambda_{1}=0$; therefore, all of the other eigenvectors
$\varphi_{j}$, being orthogonal to $\varphi_{1}$, must have mean
zero. Hence, since $\lambda_{1}=0$,
\[
	h (z,z')=\sum_{j=2}^{m} \lambda_{j}\varphi_{j} (z)\varphi_{j} (z'),
\]
and so
\begin{align}\label{eq:spectralRepresentation}
	T_{n}&=\sum_{k=2}^{m}\lambda_{k}\sum_{i=1}^{n}\sum_{j=1}^{n}
	\varphi_{k} (Z_{i}) \varphi_{k} (Z_{j})\\
\notag 	&=\sum_{k=2}^{m}\lambda_{k} \left(\sum_{i=1}^{n}\varphi_{k}
		(Z_{i}) \right)^{2}. 
\end{align}
Since each $\varphi_{k}$ has mean zero and variance $1$ relative to
$\pi$, the central limit theorem for Markov chains implies that as $n
\rightarrow \infty$,
\begin{equation}\label{eq:eigfnCLT}
	\frac{1}{\sqrt{n}}\sum_{i=1}^{n}\varphi_{k} (Z_{i})
	\Longrightarrow \text{Normal}
	(0,\sigma^{2}_{{k}}), 
\end{equation}
with limiting variances $\sigma_{k}^{2}\geq 0$. In fact, these
normalized sums 
converge \emph{jointly}\footnote{Note, however, that the normalized
sums in \eqref{eq:eigfnCLT} need not be asymptotically independent for
different $k$, despite the fact that the different functions
$\varphi_{k}$ are uncorrelated relative to $\pi$. This is because the
arguments $Z_{i}$ are serially correlated: in particular, even though
$\varphi_{k} (Z_{i})$ and $\varphi_{l} (Z_{i})$ are uncorrelated, the
random variables $\varphi_{k} (Z_{i})$ and $\varphi_{l} (Z_{i+1})$
might well be correlated.} (for $k=2,3,\dotsc ,m$) to a multivariate
normal distribution with marginal variances $\sigma_{k}^{2}\geq 0$.
The result therefore
follows from the spectral representation
\eqref{eq:spectralRepresentation}.
\end{proof}

\begin{corollary}\label{corollary:non-centered}
If $h$ is not centered, then with $\sigma^{2}>0$ as defined in
\eqref{eq:limitVariance}, 
\begin{equation}\label{eq:U-CLT}
	W_{n}/n^{3/2} \Longrightarrow 
	\text{Normal} (0,\sigma^{2}) .
\end{equation}
\end{corollary}

\begin{proof}
Recall that $W_{n}= (T_{n}-D_{n})/2$.
By the ergodic theorem, $\lim_{n \rightarrow \infty }D_{n}/n =Eh
(Z_{1},Z_{1})$ almost surely, so $D_{n}/n^{3/2}\Rightarrow 0$. Hence,
by Lemma~\ref{lemma:tools} of the Appendix, it suffices to prove that
if $h$ is not centered then
\begin{equation}\label{eq:Tnondeg}
	T_{n}/n^{3/2} \Longrightarrow \text{Normal} (0,4\sigma^{2})
\end{equation}

Define the {centered} kernel  $h^{*}$ as in
\eqref{eq:kernelAdjusted}. Since the Hoeffding projection of $H^{*}$
is identically $0$,
\begin{gather*}
	T_{n}=T^{*}_{n}+2n\sum_{i=1}^{n}H (Z_{i}) \quad \text{where}\\
	T^{*}_{n}=\sum_{i=1}^{n}\sum_{j=1}^{n} h^{*} (Z_{i},Z_{j}).
\end{gather*}
Proposition~\ref{proposition:centeredKernel} implies that
$T^{*}_{n}/n$ converges in distribution, and it follows that
$T^{*}_{n}/n^{3/2}$ converges to $0$ in distribution. On the other
hand, the central limit theorem for Markov chains implies that 
\[
	n^{-3/2}\left( 2n\sum_{i=1}^{n}H (Z_{i})\right)
	\Longrightarrow  \text{Normal}
	(0,4\sigma^{2}) ,
\]
with $\sigma^{2}>0$, since by hypothesis the kernel $h$ is not
centered. The weak convergence \eqref{eq:Tnondeg} now follows by
Lemma~\ref{lemma:tools}.
\end{proof}

\subsection{Variance/covariance bounds}\label{ssec:varBounds}

To prove Theorem~\ref{theorem:U-CLT} in the general case we will
show that truncation of the kernel $h$, that is,
replacing $h=\sum_{k=1}^{\infty }h_{k}$ by
$h^{K}=\sum_{k=1}^{K}h_{k}$, has only a small effect on the
distributions of the normalized random variables $W_{n}/n^{3/2}$ when
$K$ is large. For this we will use second moment bounds. To deduce
these from the first-moment hypothesis \eqref{eq:hypExpDecay} we shall
appeal to the fact that any aperiodic, irreducible, finite-state
Markov chain is exponentially mixing. Exponential mixing is
expressed in the same manner as for the Markov chain considered in
section~\ref{ssec:mixingRate}.  For any finite subset $J\subset \zz{N}$, let
$Z_{J}= (Z_{j})_{j\in J} $ denote the
restriction of $\mathbf{Z}$ to the index set $J$, and denote by
$\mu_{J}$ the distribution of $Z_{J}$. If $I,J$ are nonoverlapping
subsets of $\zz{N}$ then both $\mu_{I\cup J}$ and $\mu_{I}\times
\mu_{J}$ are probability measures supported by $\mathcal{A}^{I\cup
J}$. If the distance between the sets $I$ and $J$ is at least $m_{*}$,
where $m_{*}$ is the smallest integer such that all entries of
$\zz{Q}^{m_{*}}$ are positive, then  $\mu_{I\cup J}$ and $\mu_{I}\times
\mu_{J}$ are mutually absolutely continuous.

\begin{lemma}\label{lemma:expMixingGeneral}
There exist  constants $C<\infty $ and $0<\delta  <1$ such that for any
two subsets $I,J\subset \zz{N}$ satisfying $\min (J)-\max (I)=m\geq
m_{*}$,
	
\[
	1-C\delta^{m} \leq \frac{d\mu_{I\cup J}}{d\mu_{I}\times
	\mu_{J}} \leq 1+C\delta^{m}.
\]
\end{lemma}

The proof is nearly identical to that of
Lemma~\ref{lemma:expSeparation}, except that the exponential
convergence bounds of Lemma~\ref{lemma:exponentialErgodicity} must be
replaced by corresponding bounds for the transition probabilities of
$\mathbf{Z}$. The corresponding bounds are gotten from the
Perron-Frobenius theorem.

For any two random variables $U,V$ denote by $\cov (U,V)=E (UV) -EUEV$
their covariance. (When $U=V$ the covariance $\cov (U,V)=\var (U)$.)

\begin{lemma}\label{lemma:varianceBounds}
 For any two pairs $i<j$ and $i'<j'$ of indices, let $\Delta =\Delta
(i,i',j,j')$ be the distance between the sets $\{i,j \}$ and $\{i',j'
\}$ (that is, the minimum distance between one of $i,j$ and one of
$i',j'$). Then for all $\Delta \geq \max (k,k')+m_{*}$,
\begin{equation}\label{eq:varianceBounds}
	|\cov( h_{k} (\tau^{i}\mathbf{Z} ,\tau^{j}\mathbf{Z}) ,h_{k'}
	(\tau^{i'}\mathbf{Z} ,\tau^{j'}\mathbf{Z}))| \leq 
	C\beta ^{k+k'-4} \varrho_{\Delta -\max(k,k')}
\end{equation}
where
\[
	\varrho_{m}=
 	\left(\frac{1+ C \delta ^{m-1}}{1- C\delta ^{m-1}}
	\right)^{5} -1 \quad \text{for} \;\; m\geq m_{*}.
\]
\end{lemma}

\begin{remark}\label{remark:covEst}
What is important is that the covariances decay exponentially in both
$k+k'$ and $\Delta$; the rates will not matter. When $\Delta \leq \max
( k,k')+m_{*}$ the bounds \eqref{eq:varianceBounds} do not
apply. However, in this case, since the functions $h_{k}$ are bounded
above in absolute value by a constant $C<\infty$ independent of $k$
(hypothesis (H1)), the Cauchy-Schwartz inequality implies
\begin{align*}
	|\cov( h_{k} (\tau^{i}\mathbf{Z} ,\tau^{j}\mathbf{Z}) ,h_{k'}
	(\tau^{i'}\mathbf{Z} ,\tau^{j'}\mathbf{Z}))|^{2}
	&= ( E h_{k} (\tau^{i}\mathbf{Z} ,\tau^{j}\mathbf{Z})h_{k'}
	(\tau^{i'}\mathbf{Z} ,\tau^{j'}\mathbf{Z})) )^{2}\\
	&\leq ( E h_{k} (\tau^{i}\mathbf{Z} ,\tau^{j}\mathbf{Z})Eh_{k'}
	(\tau^{i'}\mathbf{Z} ,\tau^{j'}\mathbf{Z})) )^{2} \\
	&\leq C^{2} E h_{k} (\tau^{i}\mathbf{Z} ,\tau^{j}\mathbf{Z})Eh_{k'}
	(\tau^{i'}\mathbf{Z} ,\tau^{j'}\mathbf{Z})) \\
	&\leq C_{*}\beta^{k+k'},
\end{align*}
the last by the first moment hypothesis \eqref{eq:hypExpDecay}.
\end{remark}

\begin{proof}
[Proof of Lemma \ref{lemma:varianceBounds}]
Since the random variables
$h_{k}$ are functions only of the first $k$ letters of
their arguments, the covariances can be calculated by 
averaging against the measures $\mu_{J\cup K}$, where 
\[
	J=[i,i+k]\cup [j,j+k] \quad \text{and} \quad K=[i',i'+k']\cup [j',j+k'].
\]
The simplest case is where $j+k<i'$; in this case the result of
Lemma \ref{lemma:expMixingGeneral} applies directly, because the sets $J$
and $K$ are separated by $m=\Delta -k$.  Since the functions
$h_{k}$ are uniformly bounded, Lemma~\ref{lemma:expMixingGeneral} implies
\[
	\frac{1- C  \beta^{m-1}}{1+ C  \beta^{m-1}}
	\leq \frac{Eh_{k} (\tau^{i}\mathbf{Z}
	,\tau^{j}\mathbf{Z})h_{k'} (\tau^{i'}\mathbf{Z}
	,\tau^{j'}\mathbf{Z})}{Eh_{k} (\tau^{i}\mathbf{Z}
	,\tau^{j}\mathbf{Z})E
	h_{k'}(\tau^{i'}\mathbf{Z},\tau^{j'}\mathbf{Z})} \leq \frac{1+ C  
	\beta^{m-1}}{1- C  \beta^{m-1}} .
\]
The inequalities in
\eqref{eq:varianceBounds} now follow, by the assumption \eqref{eq:hypExpDecay}.
(In this special case the bounds obtained are tighter that those in
\eqref{eq:varianceBounds}.)

The other cases are similar, but the  exponential ergodicity estimate
\eqref{eq:expSeparation}  must be used indirectly, since the index
sets $J$ and $K$ need not be ordered as required by
Lemma~\ref{lemma:expSeparation}. Consider, for definiteness, the case
where
\[
	i+k \leq i'\leq i'+k'\leq j\leq j +k\leq j'\leq j'+k'.
\] 
To bound the relevant likelihood ratio in this case, use the
factorization 
\begin{align*}
	\frac{d\mu_{J\cup K}}{d\mu_{J}\times
	\mu_{K}}=&\frac{d\mu_{J\cup K}}{d\mu_{J^{-}\cup K^{-}}\times
	\mu_{J^{+}\cup K^{+}}}  \\ 
	&\times \frac{d\mu_{J^{-}\cup K^{-}}\times \mu_{J^{+}\cup
	K^{+}}}{d\mu_{J^{-}}\times \mu_{K^{-}}\times \mu_{J^{+}}\times \mu_{K^{+}}} \\
	&\times \frac{d\mu_{J^{-}}\times \mu_{K^{-}}\times
	\mu_{J^{+}}\times \mu_{K^{+}}}{d\mu_{J}\times \mu_{K}}	
\end{align*}
where $J^{-}=[i,i+k]$, $J^{+}=[j,j+k]$, $K^{-}=[i',i'+k']$, and
$K^{+}=[j',j'+k']$.  For the second and third factors, use the fact
that Radon-Nikodym derivatives of product measures factor, e.g.,
\[
	\frac{d\mu_{J^{-}}\times \mu_{K^{-}}\times
	\mu_{J^{+}}\times \mu_{K^{+}}}{d\mu_{J}\times \mu_{K}}
	(x_{J},x_{K}) =\frac{d\mu_{J^{-}}\times \mu_{J^{+}}}{d\mu_{J}}
	(x_{J}) \times \frac{d\mu_{K^{-}}\times \mu_{K^{+}}}{d\mu_{K}} (x_{K})
\]
Now Lemma~\ref{lemma:expMixingGeneral} can be used
to bound each of the resulting five factors. This yields the following
bounds: 
\[
	\left(\frac{1+ C \beta^{m-1}}{1- C \beta^{m-1}}
	\right)^{5}
	\leq \frac{d\mu_{J\cup K}}{d\mu_{J}\times
	\mu_{K}} \leq 
	\left(\frac{1+ C \beta^{m-1}}{1- C \beta^{m-1}}
	\right)^{5},
\]
and so by the same reasoning as used earlier,
\[
	\left(\frac{1+ C \beta^{m-1}}{1- C \beta^{m-1}}
	\right)^{5}
	\leq \frac{Eh_{k} (\tau^{i}\mathbf{Z}
	,\tau^{j}\mathbf{Z})h_{k'} (\tau^{i'}\mathbf{Z}
	,\tau^{j'}\mathbf{Z})}{Eh_{k} (\tau^{i}\mathbf{Z}
	,\tau^{j}\mathbf{Z})E
	h_{k'}(\tau^{i'}\mathbf{Z},\tau^{j'}\mathbf{Z})} \leq 
	\left(\frac{1+ C \beta^{m-1}}{1- C \beta^{m-1}}
	\right)^{5}.
\]

The remaining cases can be handled in the same manner.
\end{proof}

\begin{corollary}\label{corollary:varEstimates}
There exist $C,C'<\infty$ such that 
for all     $n\geq 1$ and all $1\leq K \leq L\leq \infty$,
\begin{equation}\label{eq:varEstimate}
	\var  \left(\sum_{i=1}^{n} \sum_{j=i+1}^{n} 
	\sum_{k=K}^{L} h_{k}
	(\tau^{i}\mathbf{X} ,\tau^{j}\mathbf{Z}) \right)   
	\leq C n^{3} \sum_{k=K}^{\infty }\sum_{k'=K}^{\infty }\left\{ (k'+k+C')
	\beta ^{k+k'}\right\}.
\end{equation}
Consequently, for any $\varepsilon >0$ there exists $K<\infty$ such
that for all $n\geq 1$,
\begin{equation}\label{eq:epsilonVarEstimate}
	\var  \left(\sum_{i=1}^{n} \sum_{j=i+1}^{n}\sum_{k=K+1}^{\infty} h_{k}
	(\tau^{i}\mathbf{Z} ,\tau^{j}\mathbf{Z})  \right)\leq \varepsilon n^{3}.
\end{equation}
\end{corollary}

\begin{proof} 
The variance is gotten by summing the covariances of all possible
pairs of terms in the sum. Group these by size, according to the value
of $\Delta (i,i',j,j')$: for any given value of $\Delta \geq 2$, the
number of quadruples $i,i',j,j'$ in the range $[n]$ with $\Delta
(i,i',j,j')=\Delta$ is no greater than $24 n^{3}$. For each such
quadruple and any pair $k,k'$ such that $K<k\leq k'$
Lemma~\ref{lemma:varianceBounds} implies that if $\Delta \geq k+m_{*}$
then
\[
	|\cov (h_{k} (\tau^{i}\mathbf{Z},\tau^{j}\mathbf{Z}),h_{k'}
	(\tau^{i'}\mathbf{Z},\tau^{j'}\mathbf{Z}))|\leq C \beta^{k+k'}
	\varrho_{\Delta -k'} 
\]
If $\Delta \leq m_{*}+k'$ then the crude Cauchy-Schwartz bounds of
Remark~\ref{remark:covEst}  imply that
\[
	|\cov (h_{k} (\tau^{i}\mathbf{Z},\tau^{j}\mathbf{Z}),h_{k'}
	(\tau^{i'}\mathbf{Z},\tau^{j'}\mathbf{Z}))|
	\leq C \beta^{k+k'}
\]
where $C<\infty$ is a constant independent of
$i,i',j,j',k,k'$. Summing these bounds we find that the variance on
the left side of \eqref{eq:varEstimate} is bounded by
\[
	Cn^{3}\sum_{k=K}^{L}\sum_{k=K}^{L}\beta^{k+k'} (m_{*}+k+k'
	+\sum_{\Delta =m_{*}}^{\infty}\varrho_{\Delta })
\]
Since $\varrho_{j}$ is exponentially decaying in $j$, the inner sum is
finite. This proves the inequality \eqref{eq:varEstimate}. The second
assertion now follows.
\end{proof}

\subsection{Proof of
Theorem~\ref{theorem:U-CLT}}\label{ssec:proofGeneralCase} 
Given Corollary~\ref{corollary:varEstimates} --- in particular, the
assertion \eqref{eq:epsilonVarEstimate} ---
Theorem~\ref{theorem:U-CLT} follows from the special case where all
but finitely many  of the functions $h_{k}$ are identically zero, by
Lemma~\ref{lemma:Btools} of the Appendix. To see this, observe that
under the hypotheses of Theorem~\ref{theorem:U-CLT}, the random
variable $W_{n}$ can be partitioned as 
\[
	W_{n}=W^{K}_{n}+R^{K}_{n}
\]
where
\[
	W^{K}_{n}=\sum_{i=1}^{n}\sum_{j=i+1}^{n}
	\sum_{k=1}^{K}h_{k}(\tau^{i}\mathbf{Z},\tau^{j}\mathbf{Z}) 
	\quad \text{and} \quad 
      R^{K}_{n}=\sum_{i=1}^{n}\sum_{j=i+1}^{n}
	\sum_{k=K+1}^{\infty}h_{k}(\tau^{i}\mathbf{Z},\tau^{j}\mathbf{Z}) .
\]
By Proposition~\ref{proposition:centeredKernel} and Corollary
\ref{corollary:non-centered}, for any finite $K$ the sequence
$W^{K}_{n}/n^{3/2}$ converges to a normal distribution with mean $0$
and finite (but possibly zero) variance $\sigma_{K}^{2}$. By
\eqref{eq:epsilonVarEstimate}, for any $\varepsilon >0$ there exists
$K<\infty$ such that $E|R^{K}_{n}|^{2}/n^{3}<\varepsilon$ for all
$n\geq 1$. Consequently, by Lemma~\ref{lemma:Btools},
$\sigma^{2}=\lim_{K \rightarrow \infty }\sigma_{K}^{2}$ exists and is
finite, and 
\[
	W_{n}/n^{3/2} \Longrightarrow \text{Normal} (0,\sigma^{2}).
\]

\section{Mean/Variance Calculations }\label{sec:variance}

In this section we verify that in the special case
$h_{k}=u_{k}+v_{k}$, where $u_{k}$ and $v_{k}$ are the functions
defined in section~\ref{sec:combinatorics} and $h_{1}\equiv 0$, the
constants $\kappa $ and $\sigma^{2}$ defined by
\eqref{eq:constantsMostGeneral} coincide with the values
\eqref{eq:constants}.

Assume throughout this section that
$\mathbf{X}=X_{1}X_{2}\dotsc$ and $\mathbf{X}'=X_{1}'X_{2}'\dotsb $
are two \emph{independent} stationary Markov chains  with
transition probabilities \eqref{eq:tps}, both defined on a probability
space $(\Omega ,P)$ with corresponding expectation operator $E$. Set
$h_{k}=u_{k}+v_{k}$. 
 For
each fixed (nonrandom) string
 $x_{1}x_{2}\dotsb$ of length $\geq k$  define $H_{k}=U_{k}+V_{k}$ where
\begin{align}\label{eq:conditionalMeans}
	U_{k}(x_{1}x_{2} \dotsb )&=Eu_{k} (x_{1}x_{2}\dotsb
	x_{k},X_{1}'X_{2}'\dotsb),\\
\notag 	V_{k}(x_{1}x_{2} \dotsb )&=Ev_{k} (x_{1}x_{2}\dotsb
	x_{k},X_{1}'X_{2}'\dotsb ), \quad \text{and}\\
\notag 	S_{k} (x_{1}x_{2}\dotsb )&=  U_{2} (x_{1}x_{2}) +
	\sum_{l=3}^{K} (U_{l}+V_{l})	(x_{1}x_{2}\dotsb
	x_{k}). 
\end{align}
The restrictions of $U_{k}$ and $V_{k}$ to the space of infinite
sequences  are the \emph{Hoeffding projections} of the functions $u_{k}$
and $v_{k}$ (see section~\ref{sec:strategy}).
Note that each of the functions $U_{k},V_{k},H_{k}$ depends only on
the first $k$ letters of the string $x_{1}x_{2}\dotsb$. By equations
\eqref{eq:constantsMostGeneral} of Theorem~\ref{theorem:U-CLT}, the
limit constants $\kappa $ and $\sigma^{2}$ are
\[
	\kappa =\sum_{k=2}^{\infty} EH_{k} (\mathbf{X})
	\quad \text{and} \quad 
	\sigma^{2}=\lim_{n \rightarrow \infty} \frac{1}{n}
	E\left(\sum_{i=1}^{n}\sum_{k=2}^{\infty} H_{k}
	(\tau^{i}\mathbf{X})-n\kappa \right)^{2}.
\]
We will prove (Corollary~\ref{corollary:variances}) that in the
particular case of interest here, where $h_{k}=u_{k}+v_{k}$, the random
variables $H_{k} (\tau^{i}\mathbf{X})$ and $H_{k'} (
\tau^{i'}\mathbf{X})$ are uncorrelated unless $i=i'$ and $k=k'$. It 
then follows that the terms of the sequence defining $\sigma^{2}$ are
all equal, and so
\[
	\sigma^{2}=\sum_{k=2}^{\infty } \var (H_{k} (\mathbf{X})).
\]

\begin{lemma} 
\label{lemma:conditionalMeanFormula}
For each string $x_{1}x_{2}\dotsb x_{k}$ of length $k\geq 2$ and each
index $i\leq k-1$, define $j_{i}=j_{i}(x_{1}x_{2}\dots x_{k})$ to be
the number of letters between $\bar{x}_{i}$ and $x_{i+1}$ in the
reference word $\mathcal{O}$ in the clockwise direction (see Figure~\ref{ji}). Then

\begin{equation}\label{eq:UequalV}
	V_{2}=0 \quad \text{and} \quad U_{k}=V_{k} \quad \text{for
	all} \;\; k\geq 3,
\end{equation}
and 
\begin{equation}\label{eq:conditionalMeanFormula}
	U_{k}(x_{1}x_{2}\dotsb )=
	 \frac{t(j_{1},j_{k-1})}{\ca(\ca-1)^{k-1}}
\end{equation}
where $t(a,b)=a(\ca-2 - b)+b(\ca-2-a).$

Therefore,
\begin{equation}\label{eq:H-Formula}
	S_{K} (x_{1}x_{2}\dotsb)
	=\frac{t(j_{1},j_{1})}{\ca(\ca-1)}+2\sum_{k=3}^{K}
	\frac{t(j_{1},j_{k-1})}{\ca(\ca-1)^{k-1}}. 
\end{equation}
\end{lemma}
\begin{figure}[http]
\begin{pspicture}(12,2)
\pscircle[linestyle=dashed](6,1){1}
 \rput(6,1){$\mathcal{O}$}
\rput(6.8,2){$\overline{x}_i$}\rput(7.2,0.2){${x}_{i+1}$}

\psarc[linewidth=2pt,linecolor=lightgray](6,1){1}{-50}{-30}
\psarc[linewidth=2pt,linecolor=lightgray](6,1){1}{45}{65}
\psarc[linewidth=2pt](6,1){1}{-30}{45}
\rput(9,1){$j_i(x_1x_2\dots x_ix_{i+1}\dots)$}
\end{pspicture}
\caption{The interval of length $j_i$ in $\mathcal{O}$}\label{ji}
\end{figure}
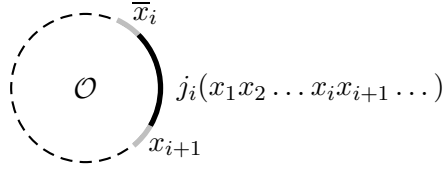
\begin{proof}
The Markov chain with transition probabilities \eqref{eq:tps} is
reversible (the transition probability matrix \eqref{eq:tps} is
symmetric), and the transition probabilities are unchanged by
inversion $a\mapsto \bar{a}$ and $a'\mapsto \bar{a}'$. Hence, the
random strings $X_{1}'X_{2}'\dotsb X_{k}'$ and
$\bar{X}_{k}'\bar{X}_{k-1}'\dotsb \bar{{X}_{1}'}$ have the same
distribution. It follows that for each $k\geq 2$,
\[
	U_{k}(x_{1}x_{2}\dotsb )=V_{k}(x_{1}x_{2}\dotsb ).
\]

Consider the case $k=2$. In order that
$u_{2}(x_{1}x_{2},X_{1}'X_{2}')\not =0$ it is necessary and sufficient
that the letters $\bar{x}_{1}\bar{X}_{1}'x_{2}X_{2}'$ occur in cyclic
order (either clockwise or counterclockwise) in the reference word
$\mathcal{O}$. For \emph{clockwise} cyclic ordering, the letter
$\bar{X}_{1}'$ must be one of the $j_{1}$ letters between
$\bar{x}_{1}$ and $x_{2}$, and $X_{2}'$ must be one of the $\ca-2-j_{1}$ letters between $x_{2}$ and $\bar{x}_{1}$. Similarly, for
\emph{counterclockwise} cyclic ordering, $\bar{X}_{1}'$ must be one of
the $\ca-2-j_{1}$ letters between $x_{2}$ and $\bar{x}_{1}$, and
$X_{2}$ one of the $j_{1}$ letters between $\bar{x}_{1}$ and
$x_{2}$. But $X_{1}'$, and hence also its inverse $\bar{X}_{1}'$, is
uniformly distributed on the $\ca$  letters, and given the value of
$\bar{X}_{1}'$ the random variable $X_{2}'$ is uniformly distributed on the
remaining $(\ca -) 1$ letters. Therefore,
\[
	U_{2} (x_{1}x_{2})=\frac{t (j_{1},j_{1})}{\ca (\ca- 1)}.
\]

The case $k\geq 3$ is similar. In order that $u_{k} (x_{1}x_{2}\dotsb
,X_{1}'X_{2}'\dotsb)$ be nonzero it is necessary and sufficient that
the strings $x_{1}x_{2}\dotsb x_{k}$ and $X_{1}'X_{2}'\dotsb X_{k}'$
differ precisely in the first and $k$th entries, and that 
the letters $\bar{x}_{1}\bar{X}_{1}'x_{2}$ occur in the same cyclic
order as the letters $x_{k}X_{k}'\bar{x}_{k-1}$. This order will be
\emph{clockwise} if and only if $\bar{X}_{1}'$ is one of the $j_{1}$
letters between $\bar{x}_{1}$ and $x_{2}$ \emph{and} $X_{k}'$ is one
of the $\ca -2-j_{k-1}$ letters between $x_{k}$ and
$\bar{x}_{k-1}$. The order will be \emph{counterclockwise} if and only
if $\bar{X}_{1}'$ is one of the $\ca -2-j_{1}$ letters between
$x_{2}$ and $\bar{x}_{2}$ \emph{and}   $X_{k}'$ is one
of the  $j_{k-1}$  letters between $\bar{x}_{k-1}$ and $x_{k}$. 
Observe that all of these possible choices will lead to reduced
words $X_{1}'x_{2}x_{3}\dotsb x_{k-1}X_{k}'$. By \eqref{eq:tps}, the
probability of one of these events occurring is
\[
	U_{k} (x_{1}x_{2}\dotsb x_{k})=\frac{t (j_{1},j_{k-1})}{\ca (\ca - 1)^{k-1}}.
\]
\end{proof}

For $i=1,2,\dotsc$, define $J_{i}=j_{i} (X_{1}X_{2}\dotsb )$ to be the
random variable obtained by evaluating the function $j_{i}$ at a
random string generated by the Markov chain, that is, $J_{i}$ is the
number of letters between $\bar{X}_{i}$ and $X_{i+1}$ in the reference
word $\mathcal{O}$ in the clockwise direction. Because $X_{i+1}$ is
obtained by randomly choosing one of the letters of $\gen$
other than $\bar{X}_{i}$, the random variable $J_{i}$ is independent
of $X_{i}$. Since these random choices are all made independently, the
following is true:

\begin{lemma}\label{lemma:J-structure}
The random variables $X_{1},J_{1},J_{2},\dotsc$ are mutually
independent,  and each $J_{i}$ has the uniform distribution on the set 
$\{0,1,2,\dotsc ,\ca-2 \}$.  Consequently,
\begin{align}\label{eq:EJ}
	EJ_{i} &= (\ca -2) /2 ,\\
\notag  	EJ_{i}^{2}&= (\ca -2) (2\ca -3)/6, \\
\notag 		EJ_{i}^{3}&= (\ca -2)^{2}(\ca - 1) /4, \\
\notag 		EJ_{i}^{4}&= (\ca -2) (2\ca -3) (3\ca^{2}-9\ca +5 )/30
\quad \text{and}\\
\notag 
 	EJ_{i}J_{i'}&= EJ_{i}EJ_{i'}= (\ca -2)^{2}/4 \quad \text{for} \;\; i\not =i'.
\end{align}
\end{lemma}
\qed

By Lemma~\ref{lemma:conditionalMeanFormula}, the conditional
expectations $U_{k}, V_{k}$ are quadratic functions  of the cycle gaps
$J_{1},J_{2},\dotsc$. Consequently, the unconditional expectations
\[
	Eu_{k} (\mathbf{X},\mathbf{X}') =EU_{k} (\mathbf{X})
\]
can be deduced from the elementary formulas of
Lemma~\ref{lemma:J-structure} by linearity of expectation.  Consider
first the case $k\geq 3$:
\begin{align*}
	\ca(\ca -1)^{k-1}EU_{k} (\mathbf{X})&=Et (J_{1},J_{k-1})\\
	 &=2E J_{1} (\ca - 2-J_{k-1})\\
	       &=2 (\ca-2)EJ_{1} - 2EJ_{1}J_{k-1}\\
	       &= (\ca-2)^{2} - (\ca-2)^{2}/2\\
	       &= (\ca -2)^{2}/2.
\end{align*}
For $k=2$:
\begin{align*}
	\ca(\ca -1 ) EU_{2} (\mathbf{X})&=Et (J_{1},J_{1})\\
	       &= 2E J_{1} (\ca-2-J_{1})\\
	      &=2 (\ca-2) EJ_{1} -2EJ_{1}^{2} \\
	      &= (\ca-2)^{2} - (\ca-2) (2\ca-3)/3\\
	      &= (\ca -2) (\ca -3)/3.
\end{align*}

\begin{corollary}\label{corollary:unconditionalMeans}
If $\mathbf{X}=X_{1}X_{2}\dotsc$ and $\mathbf{X}'=X_{1}'X_{2}'\dotsb $
are  \emph{independent} realizations of the 
stationary Markov chain with transition probabilities
\eqref{eq:tps}, then
\begin{align}\label{eq:unconditionalMean2}
	Eu_{2} (\mathbf{X},\mathbf{X}')&=EU_{2} (\mathbf{X})
	 =\frac{(\ca -2) (\ca -3)}{3\ca(\ca -1 ) } ,\\ 
\label{eq:unconditionalMeanK}
	 Eu_{k}(\mathbf{X},\mathbf{X}')&=EU_{k} (\mathbf{X}) =
	 \frac{(\ca -2)^{2}}{2 \ca(\ca -1)^{k-1}} 
	\quad \text{for}\;\; k\geq 3,
	 \quad \text{and}\\ 
\label{eq:unconditionalMeanH}
	ES_{\infty} (\mathbf{X})&=Eu_{2} (\mathbf{X},\mathbf{X}')
	+\sum_{k=3}^{\infty}E(u_{k}+v _{k}) (\mathbf{X},\mathbf{X}') 
	 =\frac{(\ca-2)}{3(\ca-1) }.
\end{align}
\end{corollary}

The variances and covariances of the random variables $U_{k}
(\mathbf{X})$ can be calculated in similar fashion, using the
independence of the cycle gaps $J_{k}$ and the moment formulas in
Lemma~\ref{lemma:J-structure}.  It is easier to work
with the scaled variables $t (J_{1},J_{k})= \ca(\ca -1)^{k}U_{k+1}$
rather than with the variables $U_{k}$, and for convenience we will
write $J_{i}^{R}=\ca-2-J_{i}$.  Note that by { definition and
Lemma~\ref{lemma:J-structure}} the random variables $J_{i}$ and
$J_{i}^{R}$ both have the same distribution (uniform on the set
$\{0,1,\dotsc ,\ca-2 \}$), and therefore also the same moments.

\medskip \noindent 
\textbf{Case 0:} If $i,j,k,m$ are distinct, or if $i=j$ and $i,k,m$
are distinct, then 
\[
	Et (J_{i},J_{j})t (J_{k},J_{m})=Et (J_{i},J_{j})Et (J_{k},J_{m}),
\]
since the random variables $J_{i},J_{j},J_{k},J_{m}$ (or in the second
case $J_{i},J_{k},J_{m}$) are independent. It follows that for any
indices $i,j,k,m$ such that and $i+k\not =j+m$, the random variables
$U_{k} (\tau^{i}\mathbf{X})$ and $U_{m} (\tau^{j}\mathbf{X})$ are
uncorrelated. (Here, as usual, $\tau$ is the forward shift operator.)

\medskip \noindent 
\textbf{Case 1:} If $i,k,m\geq 1$ are distinct then
\begin{align*}
	Et (J_{i},J_{k})t (J_{i},J_{m})&= 
	  E(J_iJ_k^R+J_kJ_i^R)(J_iJ_m^R+J_mJ_i^R)\\
	   &=EJ_{i} J_{k}^{R}J_{i}J_{m}^{R}+
	    EJ_{i}J_{k}^{R}J_{i}^{R}J_{m}+
	     EJ_{i}^{R}J_{k}J_{i}J_{m}^{R}+
	      EJ_{i}^{R}J_{k}J_{i}^{R}J_{m}\\ 
	   &= ( {(\ca-2)^{2}}/{4})
	   (EJ_{i}^{2}+EJ_{i}J_{i}^{R}+EJ_{i}^{R}J_{i}+EJ_{i}^{R}J_{i}^{R})\\
	   &=( {(\ca -2)^{2}}/{4}) E (J_{i}+J_{i}^{R})^{2}\\
	   &={(\ca -2)^{4}}/{4}\\
	   &=Et (J_{i},J_{k})Et (J_{i},J_{m}).
\end{align*}
Thus, the random variables $t (J_{i},J_{k})$ and $t (J_{i},J_{m})$ are
uncorrelated. Consequently, for all choices of $i,j,k\geq 1$ such that
$j\not =k$, the random variables $U_{j} (\tau^{i}\mathbf{X})$ and
$U_{m}(\tau^{i}\mathbf{X})$ are uncorrelated. 

\medskip \noindent 
\textbf{Case 2:} If $i\not =k$ then {
\begin{align*}
	Et (J_{i},J_{i})t (J_{i},J_{k})&=EJ_{i} J_{i}^{R}J_{i}J_{k}^{R}+
	    EJ_{i}J_{i}^{R}J_{i}^{R}J_{k}+
	     EJ_{i}^{R}J_{i}J_{i}J_{k}^{R}+
	      EJ_{i}^{R}J_{i}J_{i}^{R}J_{k}\\ 
	      &= ((\ca-2)/2) (2EJ_{i}J_{i}J_{i}^{R}+2EJ_{i}J_{i}^{R}J_{i}^{R})\\
	       &=2(\ca-2) (EJ_{i}J_i(\ca -2-J_{i}))\\
	      &=2(\ca-2) ((\ca-2)EJ_{i}^2-EJ_{i}^3)\\
	      &=(\ca-2)^{3}(\ca-3)/6\\
	      &=Et (J_{i},J_{k})Et (J_{i},J_{i})
\end{align*}}
Once again, the two random variables are uncorrelated. It follows that
for all $i\geq 1$ and $m\geq 3$ the random variables $U_{2}(\tau^{i}\mathbf{X})$
and $U_{m} (\tau^{i}\mathbf{X})$ are uncorrelated.

\medskip \noindent 
\textbf{Case 3:} If $k\geq 2$ then {
\begin{align*}
	Et (J_{1},J_{k})^{2}&= EJ_{1} J_{1}J_{k}^{R}J_{k}^{R}+
	    EJ_{1}^{R}J_{1}^{R}J_{k}J_{k}+
	   2   EJ_{1}^{R}J_{1}J_{k}^{R}J_{k}\\
		&=2 ( EJ_{1}^{2})^{2}+2 ( EJ_{1}J_{1}^{R})^{2} \\
	 &= (\ca-2)^2(2\ca-3)^2/18+(\ca-2)^2(\ca-3)^2/18
\end{align*}

and so
\begin{align*}
	\text{var} (t (J_{1},J_{k}))&=	
		    Et (J_{1},J_{k})^{2} - ( Et(J_{1},J_{k}))^{2} \\
				&= (\ca-2)^2(2\ca-3)^2/18+(\ca-2)^2(\ca-3)^2/18- (g-2)^4/4\\
			       	&= \ca^2(\ca-2)^2/36.
\end{align*}

\medskip \noindent 
\textbf{Case 4:} When $k=1$:
\begin{align*}
	Et (J_{1},J_{1})^{2}&= 4EJ_{1}J_{1}J_{1}^{R}J_{1}^{R}\\
	   &=4 ((\ca-2)^{2}EJ_{1}^{2} -2 (\ca-2)EJ_{1}^{3}+EJ_{1}^{4})\\
	   &= 2(\ca -2)(\ca -3)(\ca^2-4\ca +5)/15,
\end{align*}
so
\begin{align*}
	\text{var} (t (J_{1},J_{1}))
	 &=Et (J_{1},J_{1})^{2} - ( Et(J_{1},J_{1}))^{2} \\ 
	 &= 2(\ca -2)(\ca -3)(\ca^2-4\ca +5)/15-(\ca -2)^2(\ca - 3)^2/9\\
	 &=\ca(\ca-2)(\ca-3)(\ca+1)/45,	
\end{align*}
}
This proves:

\begin{corollary}\label{corollary:variances}
The random variables $U_{k} (\tau^{i}\mathbf{X})$, where $i\geq 0$ and
$k\geq 2$, are uncorrelated, and have variances
\begin{align}\label{eq:var-k}
	\var  (U_{k} (\tau^{i}\mathbf{X}))&=\frac{(\ca-2)^{2}}{36(\ca-1)^{2k-2}} 
 	\quad  \text{for} \quad k\geq 3,\\
\notag	\var  (U_{2} (\tau^{i}\mathbf{X}))&= 
	\frac{(\ca-2)(\ca-3)(\ca+1)}{45 \ca(\ca -1)^{2}}.
\end{align}
Consequently,
\begin{align}\label{eq:var-H}
	\var  (S_{\infty} (\tau^{i}\mathbf{X}))
		   &=\var  (U_{2} (\mathbf{X}))
		   +\sum_{k=3}^{\infty} \var  (2U_{k}
		   (\mathbf{X})) \\
\notag 	   &= \var  (U_{2} (\mathbf{X}))
		   + \lim_{K \rightarrow \infty}\sum_{k=3}^{K} \var  (2U_{k} (\mathbf{X})) \\
\notag 	 &=\frac{(\ca-2)(\ca-3)(\ca+1)}{45 \ca(\ca -1)^{2}}+\frac{ \ca  -2}{9 \ca  ( \ca  -1)^2}\\
\notag       &=\frac{( \ca  -2)( \ca  ^2-2 \ca  +2)}{45 \ca  ( \ca  -1)^2}\\
\notag       &=\frac{2 \chi  (2 \chi ^2-2 \chi +1)}{45(2 \chi-1 )^2 (\chi -1)}
		   \end{align}
\end{corollary}

\appendix
\section{Background: Probability, Markov chains, weak convergence }\label{ssec:probability}
For the convenience of the reader we shall review some of the 
terminology of the subject here (All of this is standard, and can
be found in most introductory textbooks, for instance,
\cite{billingsley:pm} and \cite{billingsley:wc}.)

A \emph{probability
space} is a measure space $(\Omega ,\mathcal{B},P)$ with total mass
$1$. Integrals with respect to $P$ are called \emph{expectations} and
denoted by the letter $E$, or by $E_{P}$ if the dependence on $P$ must be
emphasized. A \emph{random variable} is a measurable, real-valued
function on $\Omega$; similarly, a \emph{random vector} or a
\emph{random sequence} is a measurable function taking values in a
vector space or sequence space. The \emph{distribution} of a random
variable, vector, or sequence $X$ is the induced probability measure
$P\circ X^{-1}$ on the range of $X$. Most questions of interest in the
subject concern the distributions of various random objects, so the
particular probability space on which these objects are defined is
usually  not important; however, it is sometimes necessary to move to
a ``larger'' probability space (e.g., a product space) to ensure that
auxiliary random variables can be defined.  This is the case, for
instance, in sec.~\ref{sec:variance}, where independent copies
of a Markov chain are needed.

\begin{definition}\label{definition:mc}
A sequence $\dotsb ,X_{-1},X_{0},X_{1},\dotsc$ of $\gen-$valued
random variables defined on some probability space
$(\mathcal{X},\mathcal{B},P)$ is said to be a \emph{stationary Markov
chain} with \emph{stationary distribution} $\pi$ and transition
probabilities { $p (a,a')$} if
for every finite sequence $w=w_{0}w_{1}\dotsb w_{k}$ of elements of
$\gen$ and every integer $m$,
\begin{equation}\label{eq:mcDef}
	P\{X_{m+j}=w_{j} \; \text{for each} \;\; 0\leq j\leq k\}
	=\pi (w_{0}) \prod_{j=0}^{k-1} p (w_{j},w_{j+1}).
\end{equation}
\end{definition}

If $p (a,a')$ is a stochastic matrix on  set $\gen$ and $\pi$
satisfies the stationarity condition $\pi (a)=\sum_{a'}\pi (a')p
(a',a)$ then there is a probability measure on the sequence space
$\gen^{\zz{Z}}$ under which the coordinate variables form a
Markov chain with transition probabilities $p (a,a')$ and stationary
distribution $\pi$. This follows from standard measure extension
theorems -- see, e.g., \cite{billingsley:pm}, sec.~1.8.

\begin{definition}\label{definition:weakConvergence}
A sequence of random variables $X_{n}$  (not necessarily all
defined on the same probability space)  is said to converge
\emph{weakly} or \emph{in distribution} to a limit distribution $F$ on
$\zz{R}$ (denoted by $X_{n}\Rightarrow F$) if
the distributions $F_{n}$ of $X_{n}$ converge to $F$ in the weak
topology on measures, that is, if
for every bounded, continuous function $\varphi :\zz{R} \rightarrow
\zz{R}$ (or equivalently, for every continuous function $\varphi $
with compact support),
\[
	\lim_{n \rightarrow \infty} \int \varphi \,dF_{n}=\int \varphi \, dF.
\]
as $n \rightarrow \infty$. 
\end{definition}

It is also customary to write $F_{n}\Longrightarrow F$ for this
convergence, since it is really a property of the distributions.  When
the limit distribution $F$ is the point mass $\delta_{0}$ at $0$ we
may sometimes write $X_{n}\Rightarrow 0$ instead of $X_{n}\Rightarrow
\delta_{0}$. The weak topology on probability measures is metrizable;
when necessary we will denote by $\varrho$ a suitable metric.  It is
an elementary fact that weak convergence of probability distributions
on $\zz{R}$ is equivalent to the pointwise convergence of the
cumulative distribution functions at all points of continuity of the
limit cumulative distribution function. Thus,
Theorem~\ref{theorem:general} is equivalent to the assertion that the
random variables $( N (\alpha )-n^{2}\kappa )/n^{3/2}$ on the
probability spaces $(\mathcal{F}_{n},\mu_{n})$ converge in
distribution to $\Phi_{\sigma}$.

We conclude with several elementary tools of weak convergence that
will be used  repeatedly throughout the paper. First, given any
countable family $X_{n}$ of random variables, possibly defined on
different probability spaces, there exist on the Lebesgue space
$([0,1],\text{Lebesgue} )$ random variables $Y_{n}$ such that for each $n$ the
random variables $X_{n}$ and $Y_{n}$ have the same
distribution. Furthermore, the random variables $Y_{n}$ can be
constructed in such a way that if the random variables $X_{n}$
converge in distribution then the random variables $Y_{n}$ converge
pointwise on $[0,1]$ (the converse is trivial). Next, define the
\emph{total variation distance} between two probability measures $\mu$
and $\nu$ defined on a common measurable space $(\Omega ,\mathcal{B})$
by
\[
	\xnorm{\mu -\nu}_{TV}=\max (\mu (A)-\nu (A))
\]
where $A$ ranges over all measurable subsets (events) of
$\Omega$.  Total variation distance is never increased by
mapping, that is, if $T:\Omega \rightarrow \Omega ' $ is a measurable
transformation then 
\begin{equation}\label{eq:mappingPrinciple}
	\xnorm{\mu \circ T^{-1} -\nu \circ T^{-1}}_{TV} \leq \xnorm{\mu -\nu}_{TV}.
\end{equation}
Also, if $\mu $ and $\nu$ are mutually absolutely continuous, with
Radon-Nikodym derivative $d\mu /d\nu$, then
\begin{equation}\label{eq:RN-TV}
	\xnorm{\mu -\nu}_{TV}=\frac{1}{2} E_{\nu}\left| \frac{d\mu}{d\nu}-1\right|
\end{equation}
 It is easily seen that if a sequence of probability measures
$\{\mu_{n} \}_{n\geq 1}$ on $\zz{R}$ is Cauchy in total variation
distance then the sequence converges in distribution. The following
lemma is elementary:

\begin{lemma}\label{lemma:tools}
Let $X_{n}$ and $Y_{n}$ be two sequences of random variables, all
defined on a common probability space, let $a_{n}$ be a
sequence of scalars, and fix $r>0$. Denote by $F_{n}$ and $G_{n}$ the
distributions of $X_{n}$ and $Y_{n},$ respectively. Then the equivalence
\begin{equation}\label{eq:slutsky}
	\frac{Y_{n}-a_{n}}{n^{r}}{\Longrightarrow} F
	\quad   \text{if and only if }\quad 
	\frac{X_{n}-a_{n}}{n^{r}}{\Longrightarrow}  F
\end{equation}
holds if either 
\begin{align}\label{eq:toolA}
	( X_{n}-Y_{n})/n^{r}&\Longrightarrow 0 
	\quad \text{or}\\ 
\label{eq:toolB}
	\xnorm{F_{n}-G_{n}}_{TV} &\longrightarrow 0
\end{align}
as $n \rightarrow \infty$.  Furthermore, \eqref{eq:toolB}
implies \eqref{eq:toolA}.
\end{lemma}

The following lemma is an elementary consequence of Chebyshev's
inequality and the definition of weak convergence. 

\begin{lemma}\label{lemma:Btools}
Let $X_{n}$ be a sequence of random variables. Suppose that for every
$\varepsilon >0$ there exist random variables $X^{\varepsilon }_{n}$
and $R^{\varepsilon }_{n}$ such that 
\begin{gather}\label{eq:approxX}
	X_{n}=X^{\varepsilon }_{n}+R^{\varepsilon}_{n},\\
\notag 	X^{\varepsilon }_{n} \Longrightarrow 
	\text{Normal}(0,\sigma_{\varepsilon }^{2}), \quad \text{and}\\ 
\notag 	E|R^{\varepsilon }_{n}|^{2}\leq \varepsilon .
\end{gather}
Then $\lim_{\varepsilon  \rightarrow 0}\sigma_{\varepsilon
}^{2}:=\sigma^{2}\geq 0$ exists and is finite, and
\begin{equation}\label{eq:normalConvergence}
	X_{n}\Longrightarrow \text{Normal} (0,\sigma^{2}).
\end{equation}
\end{lemma}

\bibliographystyle{plain}
\bibliography{mainbib}

\end{document}